\documentclass[a4paper,11pt]{amsart}

\usepackage{amsmath, amsthm, amsfonts, amssymb}
\usepackage{mathtools}
\usepackage{enumerate}
\usepackage{tensor}
\usepackage[bookmarks=false]{hyperref} 

\usepackage{color}
\usepackage[colorinlistoftodos]{todonotes}

\title{Von Neumann equivalence and properly proximal groups}
\author{Ishan Ishan}
\author{Jesse Peterson}
\author{Lauren Ruth}

\address{Department of Mathematics, University of Nebraska-Lincoln, 203 Avery Hall, PO BOX 880130, Lincoln, NE 68588}
\email{fishan2@unl.edu}

\address{Department of Mathematics, Vanderbilt University, 1326 Stevenson Center, Nashville, TN 37240, USA}
\email{jesse.d.peterson@vanderbilt.edu}

\address{Department of Mathematics, Mercy College, 555 Broadway, Dobbs Ferry, NY 10522}
\email{lruth@mercy.edu}

\thanks{J.P. was supported in part by NSF Grant DMS \#1801125 and NSF FRG Grant \#1853989}

\newtheorem{thm}{Theorem}[section]
\newtheorem{prop}[thm]{Proposition}
\newtheorem{cor}[thm]{Corollary}
\newtheorem{lem}[thm]{Lemma}
\theoremstyle{definition}
\newtheorem{defn}[thm]{Definition}
\newtheorem{defn/lem}[thm]{Definition/Lemma}

\newtheorem{examp}[thm]{Example}

\newtheorem{problem}{Problem}[]

\newcommand{\ovt}{\, \overline{\otimes}\,}
\newcommand{\oovt}[1]{\, \overline{\otimes}_{#1}\,}
\newcommand{\oht}[1]{\, \hat{\otimes}_{#1}\,}

\newcommand{\op}{{\rm op}}
\newcommand{\actson}{{\, \curvearrowright \,}}
\newcommand{\aactson}[1]{{\, \curvearrowright^{#1} \,}}

\makeatletter
\DeclareRobustCommand\frownotimes{\mathbin{\mathpalette\frown@otimes\relax}}
\newcommand{\frown@otimes}[2]{%
  \vbox{
    \ialign{##\cr
      \hidewidth$\m@th#1{}_\frown$\kern-\scriptspace\hidewidth\cr
      \noalign{\nointerlineskip\kern-1pt}
      $\m@th#1\otimes$\cr
    }%
  }%
}
\makeatother

\begin{document}
\begin{abstract}
We introduce a new equivalence relation on groups, which we call von Neumann equivalence, that is coarser than both measure equivalence and $W^*$-equivalence. We introduce a general procedure for inducing actions in this setting and use this to show that many analytic properties, such as amenability, property (T), and the Haagerup property, are preserved under von Neumann equivalence. We also show that proper proximality, which was defined recently in \cite{BIP18} using dynamics, is also preserved under von Neumann equivalence. In particular, proper proximality is preserved under both measure equivalence and $W^*$-equivalence, and from this we obtain examples of non-inner amenable groups that are not properly proximal.
\end{abstract}

\maketitle


\section{Introduction}

Two countable groups $\Gamma$ and $\Lambda$ are measure equivalent if they have commuting measure-preserving actions on a $\sigma$-finite measure space $(\Omega, m)$ such that the actions of $\Gamma$ and $\Lambda$ individually admit a finite-measure fundamental domain. This notion was introduced by Gromov in \cite[0.5.E]{Gr93} as an analogy to the topological notion of being quasi-isometric for finitely generated groups. The basic example of measure equivalent groups is when $\Gamma$ and $\Lambda$ are lattices in the same locally compact group $G$. In this case, $\Gamma$ and $\Lambda$ act on the left and right of $G$ respectively, and these actions preserve the Haar measure on $G$. 

For certain classes of groups, measure equivalence can be quite a coarse equivalence relation. For instance, the class of countable amenable groups splits into two measure equivalence classes, those that are finite, and those that are countably infinite \cite{dye1, dye2, OrWe80}. Amenability is preserved under measure equivalence, as are other (non)-approximation type properties such as the Haagerup property or property (T).  Outside the realm of amenable groups, there are a number of powerful invariants to distinguish measure equivalence classes --- for example, Gaboriau's celebrated result that measure equivalent groups have proportional $\ell^2$-Betti numbers \cite{Ga00}.  There are also a number of striking rigidity results, such as Furman's work in \cite{Fu99B, Fu99A} where he builds on the superrigidity results of Margulis \cite{Ma75} and Zimmer \cite{Zi84}; and Kida's work in \cite{Ki10, Ki11}, where he considers measure equivalence for mapping class groups and for classes of amalgamated free product groups. 

If $\Gamma \actson (X, \mu)$ is a free probability measure-preserving action on a standard measure space, then associated to the action is its orbit equivalence relation, where equivalence classes are defined to be the orbits of the action. If $\Lambda \actson (Y, \nu)$ is another free probability measure-preserving action, then the actions are orbit equivalent if there is an isomorphism $\theta: X \to Y$ of measure spaces that preserves the orbit equivalence relations, i.e., $\theta( \Gamma \cdot x) = \Lambda \cdot \theta(x)$, for each $x \in X$. If $E \subset X$ is a positive measure subset, then one can also consider the restriction of the orbit equivalence relation to $E$. The two actions are stably orbit equivalent if there exist positive measure subsets $E \subset X$ and $F \subset Y$ such that the restricted equivalence relations are measurably isomorphic. A fundamental result in the study of measure equivalence is that two groups are measure equivalent if and only if they admit free probability measure-preserving actions that are stably orbit equivalent \cite[Section 3]{Fu99B} \cite[$P_{\rm ME}5$]{Ga05}. Moreover, in this case one can take the actions to be ergodic.

Also associated to each probability measure-preserving action $\Gamma \actson (X, \mu)$ is the Murray-von Neumann crossed product von Neumann algebra $L^\infty(X, \mu) \rtimes \Gamma$ \cite{MuvN37}. This is the von Neumann subalgebra of $\mathcal B(L^2(X, \mu) \ovt \ell^2 \Gamma)$ that is generated by a copy of $L^\infty(X, \mu)$ acting on $L^2(X, \mu)$ by pointwise multiplication, together with a copy of the group $\Gamma$ acting diagonally by $\sigma_\gamma \otimes \lambda_\gamma$, where $\sigma_\gamma$ is the Koopman representation $\sigma_\gamma(f) = f \circ \gamma^{-1}$ and $\lambda_\gamma$ is the left regular representation. The crossed product $L^\infty(X, \mu) \rtimes \Gamma$ is a finite von Neumann algebra with a normal faithful trace given by the vector state corresponding to $1 \otimes \delta_e \in L^2(X, \mu) \ovt \ell^2 \Gamma$, and if the action is free then this will be a factor if and only if the action is also ergodic, in which case $L^\infty(X, \mu)$ is a Cartan subalgebra of the crossed product. Non-free actions are also of interest in this setting. In particular, in the case when $(X, \mu)$ is trivial, this gives the group von Neumann algebra $L\Gamma$, which is a factor if and only if $\Gamma$ is ICC, i.e., every non-trivial conjugacy class in $\Gamma$ is infinite \cite{MuvN43}.

A celebrated result of Singer shows that two free ergodic probability measure-preserving actions $\Gamma \actson (X, \mu)$ and $\Lambda \actson (Y, \nu)$ are stably orbit equivalent if and only if their von Neumann crossed products are stably isomorphic in a way that preserves the Cartan subalgebras \cite{Si55}.  Specifically, Singer showed that if $E \subset X$ and $F \subset Y$ are positive measure subsets and $\theta: E \to F$ is a measure space isomorphism, then $\theta$ preserves the orbit structure almost everywhere if and only if there exists an isomorphism of von Neumann algebras 
\[
\tilde \theta: 1_F (L^\infty(Y, \nu) \rtimes \Lambda) 1_F \to 1_E( L^\infty(X, \mu) \rtimes \Gamma) 1_E
\]  
such that $\tilde \theta(f) = f \circ \theta$ for all $f \in L^\infty(F, \nu_{| F})$. 

Singer's result shows that the study of measure equivalence is closely connected to the study of finite von Neumann algebras, and there have been a number of instances where techniques from one field have been used to settle long-standing problems in the other. This exchange of ideas has especially thrived since the development of Popa's deformation/rigidity theory; see for instance \cite{Po06A, Po06B, Po06C, Po07, Po08}, or the survey papers \cite{Po07B, Va07, Va10, Io13, Io18}, and the references therein.

Two groups $\Gamma$ and $\Lambda$ are $W^*$-equivalent if they have isomorphic group von Neumann algebras, i.e.,  $L\Gamma \cong L\Lambda$. This is somewhat analogous to measure equivalence (although a closer analogy is made between measure equivalence and virtual $W^*$-equivalence, which for ICC groups asks for $L\Gamma$ and $L\Lambda$ to be virtually isomorphic in the sense that each factor is stably isomorphic to a finite index subfactor in the other factor \cite[Section 1.4]{Po86}) and both equivalence relations preserve many of the same ``approximation type'' properties. These similarities led Shlyakhtenko to ask whether measure equivalence implied $W^*$-equivalence in the setting of ICC groups. It was shown in \cite{ChIo11} that this is not the case, although the converse implication of whether $W^*$-equivalence implies measure equivalence is still open.

As with measure equivalence, we have a single $W^*$-equivalence class of ICC countably infinite amenable groups \cite{Co76}, which shows that $W^*$-equivalence is quite coarse. Yet there do exist countable ICC groups that are not $W^*$-equivalent to any other non-isomorphic group \cite{IoPoVa13, BeVa14, Be15, ChIo18}.

Returning to discuss measure equivalence, if $\Gamma$ and $\Lambda$ have commuting actions on $(\Omega, m)$ and if $F \subset \Omega$ is a Borel fundamental domain for the action of $\Gamma$, then on the level of function spaces, the characteristic function $1_F$ gives a projection in $L^\infty(\Omega, m)$ such that the collection $\{ 1_{\gamma F} \}_{\gamma \in \Gamma}$ forms a partition of unity, i.e., $\sum_{\gamma \in \Gamma} 1_{\gamma F} = 1$. This notion generalizes quite nicely to the non-commutative setting where we will say that a fundamental domain for an action on a von Neumann algebra $\Gamma \aactson{\sigma} \mathcal M$ is a projection $p \in \mathcal M$ such that $\sum_{\gamma \in \Gamma}\sigma_\gamma(p) = 1$, where the convergence is in the strong operator topology. 

Using this perspective for a fundamental domain we may then generalize the notion of measure equivalence by simply considering actions on non-commutative spaces. 
\begin{defn}\label{defn:vne}
Two groups $\Gamma$ and $\Lambda$ are von Neumann equivalent, written $\Gamma \sim_{vNE} \Lambda$, if there exists a von Neumann algebra $\mathcal M$ with a semi-finite normal faithful trace ${\rm Tr}$ and commuting, trace-preserving, actions of $\Gamma$ and $\Lambda$ on $\mathcal M$ such that the $\Gamma$ and $\Lambda$-actions individually admit a finite-trace fundamental domain.
\end{defn}
The proof of transitivity for measure equivalence is adapted in Proposition~\ref{prop:vnetransitive} below to show that von Neumann equivalence is a transitive relation. It is also clearly reflexive and symmetric, so that von Neumann equivalence is indeed an equivalence relation. 

Von Neumann equivalence is clearly implied by measure equivalence, and, in fact, von Neumann equivalence is also implied by $W^*$-equivalence. Indeed, if $\theta: L\Gamma \to L\Lambda$ is a von Neumann algebra isomorphism, then we may take $\mathcal M = \mathcal B(\ell^2 \Lambda)$ with the trace-preserving action $\sigma: \Gamma \times \Lambda \to {\rm Aut}(\mathcal M)$ given by $\sigma_{(s, t)}(T) = \theta(\lambda_s) \rho_t T \rho_t^* \theta(\lambda_s^*)$, where $\rho: \Lambda \to \mathcal U(\ell^2 \Lambda)$ is the right regular representation, which commutes with operators in $L\Lambda$. It is then not difficult to see that the rank one projection $p$ onto the subspace $\mathbb C \delta_e$ is a common fundamental domain for the actions of both $\Gamma$ and $\Lambda$. In fact, we'll show below that virtual $W^*$-equivalence also implies von Neumann equivalence.

We introduce below a general induction procedure for inducing representations via von Neumann equivalence from $\Lambda$ to $\Gamma$, and using these induced representations we show that some of the properties that are preserved for measure equivalence and $W^*$-equivalence are also preserved for von Neumann equivalence. 

\begin{thm}\label{thm:vneamenaT}
Amenability, property (T), and the Haagerup property are all von Neumann equivalence invariants. 
\end{thm}

A group $\Gamma$ is properly proximal if there does not exist a left-invariant state on the $C^*$-algebra $(\ell^\infty \Gamma/ c_0 \Gamma )^{\Gamma_r}$ consisting of elements in $\ell^\infty \Gamma/ c_0 \Gamma$ that are invariant under the right action of the group. Properly proximal groups were introduced in \cite{BIP18}, where a number of classes of groups were shown to be properly proximal, including non-elementary hyperbolic groups, convergence groups, non-amenable bi-exact groups, groups admitting proper $1$-cocycles into non-amenable representations, and lattices in non-compact semi-simple Lie groups of arbitrary rank. It is also shown that the class of properly proximal groups is stable under commensurability up to finite kernels, and it was then asked if this class was also stable under measure equivalence \cite[Question 1(b)]{BIP18}. 

Proper proximality also has a dynamical formulation \cite[Theorem 4.3]{BIP18}, and using this, together with our induction technique applied to isometric representations on dual Banach spaces, we show that the class of properly proximal groups is not only closed under measure equivalence but also under von Neumann equivalence.

\begin{thm}\label{thm:vnepropprox}
If $\Gamma \sim_{vNE} \Lambda$ then $\Gamma$ is properly proximal if and only if $\Lambda$ is properly proximal.
\end{thm}

An example of Caprace, which appears in Section 5.C of \cite{DuTDWe18}, shows that the class of inner amenable groups is not closed under measure equivalence. Specifically, if $p$ is a prime and $F_p$ denotes the finite field with $p$ elements, then the group $SL_3(F_p[t^{-1}]) \ltimes F_p[t, t^{-1}]^3$ is not inner amenable, although it is measure equivalent to the inner amenable group $(SL_3(F_p[t^{-1}]) \ltimes F_p[t^{-1}]^3 ) \times F_p[t]^3$. Using the previous theorem, we then answer another question from \cite{BIP18} by providing $SL_3(F_p[t^{-1}]) \ltimes F_p[t, t^{-1}]^3$ as an example of a non-inner amenable group that is also not properly proximal. As a corollary of our result and the example of Caprace, we see that there are uncountably many examples of non-inner amenable groups that are not properly proximal --- for example, all groups of the form $\left( SL_3(F_p[t^{-1}]) \ltimes F_p[t, t^{-1}]^3 \right) \times G$, where $G$ is any non-inner amenable group.

The notion of von Neumann equivalence also admits a generalization in the setting of finite von Neumann algebras. 
\begin{defn}\label{defn:vnefactor}
Two finite von Neumann algebras $M$ and $N$ are von Neumann equivalent, written $M \sim_{vNE} N$, if there exists a semi-finite von Neumann algebra $\mathcal M$ containing commuting copies of $M$ and $N^{\rm op}$ such that we have intermediate standard representations $M \subset \mathcal B(L^2(M)) \subset \mathcal M$ and $N^{\rm op} \subset \mathcal B(L^2 (N)) \subset \mathcal M$ satisfying the property that finite-rank projections in $\mathcal B(L^2(M))$ and $\mathcal B(L^2(N))$ are finite projections in $\mathcal M$. 
\end{defn}
We show in Section~\ref{sec:vnefinite} that this does indeed give an equivalence relation, which is coarser than the equivalence relation given by virtual isomorphism. Moreover, if $\mathcal M$ is a factor then we can associate an index $[M: N]_{\mathcal M}$, which is given by 
\[
[M: N]_{\mathcal M} = {\rm Tr}(p)/{\rm Tr}(q),
\] 
where ${\rm Tr}$ is a trace on $\mathcal M$ and $p$ and $q$ are rank $1$ projections in $\mathcal B(L^2(M))$ and $\mathcal B(L^2(N))$ respectively. The connection to von Neumann equivalence for groups is given by the following theorem:

\begin{thm}\label{thm:vnequivalencegp}
If $\Gamma$ and $\Lambda$ are countable groups, then $\Gamma \sim_{vNE} \Lambda$ if and only if $L\Gamma \sim_{vNE} L\Lambda$.
\end{thm}

We show in Theorem~\ref{thm:vneindexgp} that the set of indices for factorial self von Neumann couplings forms a subgroup $\mathcal I_{vNE}(M) < \mathbb R^*_+$, which we call the index group of $M$. If $M$ is a factor then we show that the index group contains the square of the fundamental group of $M$. The fact that we have the square of the fundamental group instead of the fundamental group itself agrees with phenomena predicted by Connes and Shlyakhtenko in \cite[Theorem 2.4]{CoSh05} and leaves open the possibility that Gaboriau's theorem implying proportional $\ell^2$-Betti numbers could still hold in the setting of von Neumann equivalence. However, we make no attempt to achieve this result here. 

We also show that for a countable ICC group $\Gamma$ there is a connection between the index group of $L\Gamma$ and the class $\mathcal S_{eqrel}(\Gamma)$ studied by Popa and Vaes in \cite{PoVa10}, which consists of fundamental groups for equivalence relations associated to free, ergodic, probability measure-preserving actions of $\Gamma$. Specifically, we show in Corollary~\ref{cor:subgpindex} that $\mathcal I_{vNE}(L\Gamma)$ contains the group generated by all the groups in $\mathcal S_{eqrel}(\Gamma)$.

For the reader who may be more familiar with techniques coming from measured group theory, we end this article with an appendix where we give a direct proof in the measure equivalence setting that proper proximality is a measure equivalence invariant. 

\emph{Acknowledgments}. The authors would like to thank the anonymous referees for their useful suggestions improving the correctness and readability of the article.

\section{Preliminaries and notation}

The main techniques we use in this article involve von Neumann algebras endowed with semi-finite normal traces. We briefly discuss some of the facts regarding semi-finite von Neumann algebras that we will use in the sequel. We refer the reader to  \cite{Ta02} for proofs of these facts.

\subsection{Semi-finite traces} \label{sec:semifinite}

Let $\mathcal M$ be a semi-finite von Neumann algebra with a semi-finite normal faithful trace ${\rm Tr}$. We let $\mathcal M_+$ denote the set of positive operators in $\mathcal M$. We set $\mathfrak n_{\rm Tr} = \{ x \in \mathcal M \mid {\rm Tr}(x^*x) < \infty \}$, and $\mathfrak m_{\rm Tr} = \{ \sum_{j = 1}^n x_j^* y_j \mid x_j, y_j \in \mathfrak n_{\rm Tr}, 1 \leq j \leq n \}$. Both $\mathfrak n_{\rm Tr}$ and $\mathfrak m_{\rm Tr}$ are ideals in $\mathcal M$, and the trace ${\rm Tr}$ gives a $\mathbb C$-valued linear functional on $\mathfrak m_{\rm Tr}$, which is called the definition ideal of ${\rm Tr}$.

We let $L^1(\mathcal M, {\rm Tr})$ denote the completion of $\mathfrak m_{\rm Tr}$ under the norm $\| a \|_1 = {\rm Tr}( | a |)$, and then the bilinear form $\mathcal M \times \mathfrak m_{\rm Tr} \ni (x, a) \mapsto {\rm Tr}(xa)$ extends to the duality between $\mathcal M$ and $L^1(\mathcal M, {\rm Tr})$ so that we may identify $L^1(\mathcal M, {\rm Tr})$ with $\mathcal M_*$. 

We let $L^2(\mathcal M, {\rm Tr})$ denote the Hilbert space completion of $\mathfrak n_{\rm Tr}$ under the inner product $\langle a, b \rangle_2 = {\rm Tr}(b^* a)$. Left multiplication of $\mathcal M$ on $\mathfrak n_{\rm Tr}$ then induces a normal faithful representation of $\mathcal M$ in $\mathcal B(L^2(\mathcal M, {\rm Tr}))$, which is called the standard representation. 

Restricting the conjugation operator from $\mathcal M$ to $\mathfrak n_{\rm Tr}$ induces an anti-linear isometry $J : L^2(\mathcal M, {\rm Tr}) \to L^2(\mathcal M, {\rm Tr})$, and we have $J \mathcal M J = \mathcal M' \cap \mathcal B(L^2(\mathcal M, {\rm Tr}))$. The von Neumann algebra $J \mathcal M J$ is canonically isomorphic to the opposite von Neumann algebra $\mathcal M^{\rm op}$ via the map $\mathcal M^{\rm op} \ni x^{\rm op} \mapsto J x^* J$. We also have the induced trace on $\mathcal M^{\rm op}$ given by ${\rm Tr}(x^{\rm op}) = {\rm Tr}(x)$. 

If $\mathcal M$ is a semi-finite factor, then it has a unique (up to scalar multiples) normal semi-finite faithful trace. In general, if ${\rm Tr}_1$ and ${\rm Tr}_2$ are normal semi-finite traces, then there is an injective positive operator $a$ affiliated to the center $\mathcal Z(\mathcal M)$ such that ${\rm Tr}_2(x) = {\rm Tr}_1(a x)$ for all $x \in \mathcal M_+$. In particular, the map $\mathfrak n_{{\rm Tr}_2} \ni x \mapsto a^{1/2} x \in \mathfrak n_{{\rm Tr}_1}$ extends to a unitary operator from $L^2(\mathcal M, {\rm Tr}_2)$ onto $L^2(\mathcal M, {\rm Tr}_1)$ that intertwines the representations of $\mathcal M$, and also intertwines the representations of $\mathcal M^{\rm op}$. Thus, up to isomorphism, the representation $\mathcal M \subset \mathcal B(L^2(\mathcal M, {\rm Tr}))$ is independent of the choice of semi-finite normal faithful trace ${\rm Tr}$, and we may use the notation $\mathcal M \subset \mathcal B(L^2(\mathcal M))$ if we wish to emphasize this fact.

If $\mathcal H$ is a Hilbert space and we have an embedding $\mathcal B(\mathcal H) \subset \mathcal M$, then setting $P = \mathcal B(\mathcal H)' \cap \mathcal M$ we have an isomorphism $\mathcal B(\mathcal H ) \ovt P \cong \mathcal M$ that maps $T \otimes x$ to $Tx$ for $T \in \mathcal B(\mathcal H)$ and $x \in P$. Indeed, this is easy to verify in the case when $\mathcal M$ is a type I factor, and in general if we represent $\mathcal M \subset \mathcal B(\mathcal K)$, then we have $\mathcal B(\mathcal H) \subset \mathcal M \subset \mathcal B(\mathcal K) \cong \mathcal B(\mathcal H) \ovt \mathcal B(\mathcal K_0)$ for some Hilbert space $\mathcal K_0$. Thus, $\mathcal M' = \mathcal M' \cap \mathcal B(\mathcal K_0)$, so that $\mathcal M = \mathcal M'' = \mathcal B(\mathcal H) \ovt (\mathcal M \cap \mathcal B(\mathcal K_0) ) = \mathcal B(\mathcal H) \ovt P$. There also then exists a unique semi-finite normal faithful trace ${\rm Tr}_P$ on $P$ so that ${\rm Tr}_{\mathcal M} = {\rm Tr} \otimes {\rm Tr}_P$. Alternatively, this result follows directly from Ge-Kadison tensor splitting, Theorem A from \cite{GK96}.

If we have two embeddings $\theta_1, \theta_2: \mathcal B(\mathcal H) \to \mathcal M$, then $\theta_1(\mathcal B(\mathcal H))$ and $\theta_2(\mathcal B(\mathcal H))$ are conjugate by a unitary in $\mathcal M$ if and only if for some rank one projection $p \in \mathcal B(\mathcal H)$ we have that $\theta_1(p)$ and $\theta_2(p)$ are Murray-von Neumann equivalent.

\subsection{Actions on semi-finite von Neumann algebras}\label{sec:actions}

If $\Gamma$ is a discrete group and $\Gamma \aactson{\sigma} \mathcal M$ is an action that preserves the trace ${\rm Tr}$, then $\Gamma$ preserves the $\| \cdot \|_1$-norm on $\mathfrak m_{\rm Tr}$ and hence the action extends to an action by isometries on $L^1(\mathcal M, {\rm Tr})$, and the dual of the action on $L^1(\mathcal M, {\rm Tr})$ agrees with the action on $\mathcal M$. 

Restricted to $\mathfrak n_{\rm Tr}$ the action is also isometric with respect to $\| \cdot \|_2$ and hence gives a unitary representation in $\mathcal U(L^2(\mathcal M, {\rm Tr}))$, which is called the Koopman representation and denoted by $\sigma^0: \Gamma \to  \mathcal U(L^2(\mathcal M, {\rm Tr}))$. Note that considering $\mathcal M \subset \mathcal B(L^2(\mathcal M, {\rm Tr}))$ via the standard representation, we have that the action $\sigma: \Gamma \to {\rm Aut}(\mathcal M, {\rm Tr})$ becomes unitarily implemented via the Koopman representation, i.e., for $x \in \mathcal M$ and $\gamma \in \Gamma$ we have $\sigma_\gamma(x) = \sigma_\gamma^0 x \sigma_{\gamma^{-1}}^0$. 

The crossed product von Neumann algebra $\mathcal M \rtimes \Gamma$ is defined to be the von Neumann subalgebra of $\mathcal B(L^2(\mathcal M, {\rm Tr}) \ovt \ell^2 \Gamma)$ generated by $\mathcal M \ovt \mathbb C$ and $\{ \sigma_\gamma^0 \otimes \lambda_\gamma \mid \gamma \in \Gamma \}$. We use the notation $u_\gamma = \sigma_\gamma^0 \otimes \lambda_\gamma$. Note that by Fell's absorption principle, the representation $\Gamma \ni \gamma \mapsto u_\gamma \in \mathcal M \rtimes \Gamma$ is conjugate to a multiple of the left regular representation and hence generates a copy of the group von Neumann algebra $L\Gamma$. 

If $P_e$ denotes the rank one projection onto $\mathbb C \delta_e \subset \ell^2 \Gamma$, then we have a canonical conditional expectation from $\mathcal B(L^2(\mathcal M, {\rm Tr}) \ovt \ell^2 \Gamma)$ onto $\mathcal B(L^2(\mathcal M, {\rm Tr}))$ given by $T \mapsto (1 \otimes P_e) T (1 \otimes P_e)$ and then identifying $\mathcal B(L^2(\mathcal M, {\rm Tr}))$ with $\mathcal B(L^2(\mathcal M, {\rm Tr})) \otimes \mathbb C P_e$. Restricting this to $\mathcal M \rtimes \Gamma$ gives a conditional expectation $E_{\mathcal M}: \mathcal M \rtimes \Gamma \to \mathcal M$. The trace on $\mathcal M$ then extends to a faithful normal semi-finite trace on $\mathcal M \rtimes \Gamma$ given by ${\rm Tr}(x) = {\rm Tr} \circ E_{\mathcal M}(x)$. 

If we have a subgroup $\Gamma_0 < \Gamma$ and a $\Gamma_0$-invariant von Neumann subalgebra $\mathcal M_0$, then the von Neumann algebra generated by $\mathcal M_0$ and $\Gamma_0$ is canonically isomorphic to the crossed product $\mathcal M_0 \rtimes \Gamma_0$, and so we have a canonical embedding of crossed products $\mathcal M_0 \rtimes \Gamma_0 \subset \mathcal M \rtimes \Gamma$.

A specific example of the crossed product construction that we will use below is when we consider $\ell^\infty \Gamma$ with its trace coming from counting measure, and the action of $\Gamma \aactson{L} \ell^\infty \Gamma$ is given by right multiplication $L_\gamma(f) (x) = f(x\gamma)$. In this case, by considering a Fell unitary, we obtain an isomorphism $\theta: \ell^\infty \Gamma \rtimes \Gamma \to \mathcal B(\ell^2 \Gamma)$ such that $\theta(f)$ is the multiplication operator by $f$ for $f \in \ell^\infty \Gamma$, while for $\gamma \in \Gamma$ we have $\theta(u_\gamma) = \lambda_\gamma$ gives the left-regular representation. 

\subsection{The basic construction}

If $(\mathcal M, {\rm Tr})$ is a von Neumann algebra with a semi-finite normal faithful trace ${\rm Tr}$, then conjugation on $\mathfrak n_{\rm Tr}$ induces an anti-linear isometry $J: L^2(\mathcal M, {\rm Tr}) \to L^2(\mathcal M, {\rm Tr})$, and we have $\mathcal M' = J \mathcal M J$. If $\mathcal N \subset \mathcal M$ is a von Neumann subalgebra then the basic construction is the von Neumann algebra 
\[
\langle \mathcal M, \mathcal N \rangle := (J \mathcal N J)' \subset \mathcal B(L^2(\mathcal M, {\rm Tr})).
\] 
If $\mathcal N$ is semi-finite, the so is $\langle \mathcal M, \mathcal N \rangle$.

\subsection{Measurable functions into separable Banach spaces}

If $(X, \mu)$ is a standard measure space and $E$ is separable Banach space, then we let $L^1(X, \mu; E)$ denote the space of measurable functions $f: X \to E$ such that $\int \| f(x) \| \, d\mu(x) < \infty$, where we identify two functions if they agree almost everywhere. We have an isometric isomorphism $L^1(X, \mu) \oht{} E \to L^1(X, \mu; E)$, which takes an elementary tensor $f \otimes a$ to the function $x \mapsto f(x) a$; here $\oht{}$ represents the Banach space projective tensor product. If $E = (E_*)^*$ is dual to a separable Banach space then we let $L^\infty_{w^*}(X, \mu; E)$ denote the space of essentially bounded functions that are Borel with respect to the weak$^*$-topology restricted to some ball in $E$ that contains almost every point in the range of $f$, where we identify two functions if they agree almost everywhere. Note that since $E_*$ is separable, the weak$^*$-topology in $E$ is compact and metrizable when restricted to any closed ball. 

If $K \subset E$ is a weak$^*$-compact subset, then we denote by $L^\infty_{w^*}(X, \mu; K)$ the subset of $L^\infty_{w^*}(X, \mu; E)$ consisting of those functions whose essential range is contained in $K$. We have an isometric isomorphism $L^\infty_{w^*}(X, \mu; E) \to (L^1(X, \mu; E_*))^*$ \cite[Section 2.2]{Mo01} given by the pairing 
\[
\langle f, g \rangle = \int \langle f(x), g(x) \rangle \, d\mu(x).
\]
Thus  we have isometric isomorphisms 
\[
L^\infty_{w^*}(X, \mu; E) \cong (L^1(X, \mu; E_*) )^* \cong \mathcal B(L^1(X, \mu), E).
\]

\begin{prop}\label{prop:abelrange}
If $E = (E_*)^*$ is a dual Banach space and $K \subset E$ is a weak$^*$-closed convex subset, then under the above isomorphism we have
\[
L^\infty_{w^*}(X, \mu; K) \cong \{ \Xi \in \mathcal B(L^1(X, \mu), E) \mid \Xi(f) \in K {\rm \ for \ all \ } f \in L^1(X, \mu)_+, \| f \|_1 = 1 \}.
\]
\end{prop}
\begin{proof}
We let $\Psi: L^\infty_{w^*}(X, \mu; E) \to \mathcal B(L^1(X, \mu), E)$ be the isomorphism from above, so that for $f \in L^\infty_{w^*}(X, \mu; E)$ and $g \in L^1(X, \mu)$ we have $\Psi(f)(g) = \int g(x) f(x) \, d\mu(x)$. 

If $f \in L^\infty_{w^*}(X, \mu; K)$ and $g \in L^1(X, \mu)_+$ with $\| g \|_1 = 1$, then as $K$ is convex and weak$^*$-closed we have $\Psi(f)(g) = \int g(x) f(x) \, d\mu(x) \in K$. On the other hand, if $f \in L^\infty_{w^*}(X, \mu; E)$ is not in $L^\infty_{w^*}(X, \mu; K)$, then choose a point $m \in E \setminus K$ that is contained in the essential range of $f$. By the Hahn-Banach separation theorem, there exists a convex weak$^*$-open neighborhood $G$ of $m$ such that $\overline{G} \cap K = \emptyset$. If we set $B = f^{-1}(G)$, then we have that $\mu(B) > 0$, and taking $g = \frac{1}{\mu(B)} 1_B$, we have $\Psi(f)(g) = \frac{1}{\mu(B)} \int_B f(x) \, d\mu(x) \in \overline{G} \subset E \setminus K$. 
\end{proof}

\subsection{Tensor products of operator spaces}

For the basic results we'll need from the theory of operator spaces and their tensor products, we refer the reader to \cite{BlMe04} or \cite{Pi03}. A (concrete) operator space is a closed subspace $E \subset \mathcal B(\mathcal H)$. Given operator spaces $E$ and $F$, and a linear map $u: E \to F$, we define linear maps $u_n: \mathbb M_n(E) \to \mathbb M_n(F)$ by setting $u_n((x_{i j} ) ) = (u(x_{i j}) )$. The map $u$ is completely bounded if the completely bounded norm $\| u \|_{\rm cb} = \sup_n \| u_n \|$ is finite. 

We denote by $CB(E, F)$ the space of all completely bounded maps from $E$ to $F$, which is a Banach space when given the completely bounded norm. We also endow $\mathbb M_n( CB(E, F))$ with the Banach space norms coming from the canonical isomorphism $\mathbb M_n( CB(E, F)) \cong CB(E, \mathbb M_n(F))$. Ruan's abstract matrix norm characterization for operator spaces shows that the norms on $\mathbb M_n( CB(E, F))$ give an operator space structure to $CB(E, F)$, i.e., $CB(E, F)$ is completely isometrically isomorphic to a concrete operator space. In particular, when $F = \mathbb C$ we obtain the dual operator space structure on $E^*$. 

Any Banach space $X$ can be endowed with an operator space structure by embedding $X$ into the $C^*$-algebra $C( (X^*)_1)$ of weak$^*$-continuous functions on the unit ball of $X^*$, and where $X$ is realized via the evaluation map. We denote this operator space structure by ${\rm min}(X)$. We may also consider the supremum of all operator space norms on $X$, and we denote this operator space structure by ${ \rm max}(X)$. We then have completely isometrically ${\rm min}(X)^* = {\rm max}(X^*)$ and ${\rm max}(X)^* = {\rm min}(X^*)$.

For a Hilbert space $\mathcal H$ there are two canonical operator space structures. The first is the Hilbert column space $\mathcal H^c$, which endows $\mathcal H$ with the operator space structure coming from the canonical isomorphism $\mathcal H \cong CB(\mathbb C, \mathcal H)$. The second is the Hilbert row space $\mathcal H^r$, which endows $\mathcal H$ with the operator space structure coming from the canonical isomorphism $\mathcal H \cong CB(\mathcal H, \mathbb C)$. As operator spaces we then have natural identifications $(\mathcal H^c)^* \cong \overline{\mathcal H}^r$ and $(\mathcal H^r)^* \cong \overline{\mathcal H}^c$. Unless otherwise stated, in the sequel we will endow any Hilbert space with its operator space structure as a Hilbert column space. 

If $E \subset \mathcal B(\mathcal H)$ and $F \subset \mathcal B(\mathcal K)$ are operator spaces, the minimal tensor product $E \otimes_{\rm min} F$ is given by the completion of the algebraic tensor product $E \otimes F \subset \mathcal B(\mathcal H \ovt \mathcal K)$. The operator space structure on $E \otimes_{\rm min} F$ is independent of the concrete representations, and we have a completely isometric embedding 
\[
E \otimes_{\rm min} F \hookrightarrow CB(F^*, E)
\] 
where $u = \sum_{k = 1}^n x_k \otimes y_k \in E \otimes F$ is associated to the map $\tilde{u}: F^* \to E$ given by $\tilde{u}(\psi) = \sum_{k = 1}^n \psi(y_k) x_k$, for $\psi \in F^*$.

If $E$ and $F$ are operator spaces, then perhaps the simplest way to describe the projective tensor product is to define it as the completion of $E \otimes F$ when we embed $E \otimes F$ into the operator space $CB(E, F^*)^*$ via the map that assigns to $x \otimes y$ the functional $CB(E, F^*) \ni T \mapsto T(x)(y)$. We denote the operator space projective tensor product of $E$ and $F$ by $E \frownotimes F$. From \cite[Proposition 5.4]{BlPa91} we then have completely isometric isomorphisms 
\begin{equation}\label{eq:dualtensor}
(E \frownotimes F)^* \cong CB(E, F^*) \cong CB(F, E^*).
\end{equation}

We note that under the identification $(E \frownotimes F)^* \cong CB(E, F^*)$, the weak$^*$-topology on bounded sets is given by pointwise weak$^*$-convergence of operators.

In this article we will be mainly interested in dual operator spaces.  We therefore find it convenient to use the notation $E$ and $F$ for operator spaces that are dual to operator spaces $E_*$ and $F_*$ respectively. Every ultraweakly closed subspace $E$ of $\mathcal B(\mathcal H)$ is a dual operator space with a canonical predual $\mathcal B(\mathcal H)_*/E_{\perp}$, where $E_{\perp}$ is the preannihilator of $E$. Conversely, if $E$ is a dual operator space, then $E$ is weak$^*$-homeomorphically completely isometric to an ultraweakly closed subspace of $\mathcal B(\mathcal H)$.

If $E \subset \mathcal B(\mathcal H)$ and $F \subset \mathcal B(\mathcal K)$ are ultraweakly closed subspaces, then the normal minimal tensor product $E \ovt F$ is the ultraweak completion of the algebraic tensor product $E \otimes F \subset \mathcal B(\mathcal H \ovt \mathcal K)$. This is independent of the concrete representations, and we have a weak$^*$-homeomorphic completely isometric embedding
\[
E \ovt F \hookrightarrow (E_* \frownotimes F_*)^* \cong CB(E_*, F) \cong CB(F_*, E).
\]
We will therefore identify $E \ovt F$ as a subspace of $CB(E_*, F)$. We note that even in the case when $F = \mathcal M$ is a von Neumann algebra, this embedding will not be surjective in general. However, it follows from \cite[Theorem 2.5]{Bl92}, \cite[Proposition 3.3]{Ru92} and \cite{Kr91} that this embedding will be surjective whenever $F = \mathcal M$ is a von Neumann algebra with the $\sigma$-weak approximation property.

\subsection{Hilbert $C^*$-modules}\label{sec:HilbertC}

We refer the reader to \cite{La95} for the basic properties of Hilbert $C^*$-modules. If $A$ is a $C^*$-algebra and $I$ is a set, then we let $\bigoplus_{i \in I} A$ denote the space of functions $( a_i )_{i \in I}$ such that $\sum_{i \in I} a_i^*a_i$ converges in $A$. This gives a Hilbert $A$-module where we have an $A$-valued inner product (linear in the second variable) given by 
\[
\langle ( a_i )_{i \in I}, (b_i)_{i \in I} \rangle_A = \sum_{i \in I} a_i^* b_i.
\] 
If $\mathcal H$ is a Hilbert space then on the algebraic tensor product $A \otimes \mathcal H$ we have an $A$-valued inner product given by $\langle a \otimes \xi, b \otimes \eta \rangle_A = \langle \eta, \xi \rangle a^* b$. This inner product extends continuously to give a Hilbert $A$-module structure to $A \otimes_{\rm min} \mathcal H$ \cite[Theorem 8.2.17]{BlMe04}, where $\mathcal H$ is endowed with its operator space structure as a column Hilbert space. Choosing a basis $\{ e_i \}_{i \in I}$ gives an identification between the Hilbert $A$-modules $A \otimes_{\rm min} \mathcal H$ and $\bigoplus_{i \in I} A$.

If $\mathcal M$ is a von Neumann algebra and $I$ is a set then we let $\overline \bigoplus_{i \in I} \mathcal M$ denote the space of functions $( a_i )_{i \in I}$ such that $\sum_{i \in I} a_i^* a_i$ is bounded. If $(a_i)_{i \in I}$, $(b_i)_{i \in I} \in \overline \bigoplus_{i \in I} \mathcal M$ then we have ultraweak convergence of the sum 
\[
\langle ( a_i )_{i \in I}, (b_i)_{i \in I} \rangle_{\mathcal M} = \sum_{i \in I} a_i^* b_i.
\]
If $\mathcal H$ is a Hilbert space, then the Hilbert $\mathcal M$-module structure on $\mathcal M \otimes_{\rm min} \mathcal H$ has a unique extension to $\mathcal M \ovt \mathcal H$ such that the inner product $\langle \cdot, \cdot \rangle_{\mathcal M}$ is separately ultraweakly continuous. In particular, $\mathcal M \ovt \mathcal H$ will be self-dual in the sense of Paschke \cite{Pa73}, \cite[Proposition 2.9]{Sc02}. Choosing a basis $\{ e_i \}_{i \in I}$ gives an identification between the Hilbert $\mathcal M$-modules $\mathcal M \ovt \mathcal H$ and $\overline \bigoplus_{i \in I} \mathcal M$. 

Dual Hilbert $\mathcal M$-modules are naturally related to normal representations of $\mathcal M$ obtained via an internal tensor product $\mathcal K \oovt{\mathcal M} L^2(\mathcal M)$. When $\mathcal M$ has a finite trace $\tau$, this is quite explicit; as it will be used in the sequel, we describe it in detail here.  Given a Hilbert $\mathcal M$-module $\mathcal K$, we obtain a scalar-valued inner product $\langle \cdot, \cdot \rangle_\tau$ on $\mathcal K$ by $\langle \xi, \eta \rangle_\tau = \tau( \langle \eta, \xi \rangle_{\mathcal M})$. The completion gives a Hilbert space ${ \mathcal K_\tau}$, and the right $\mathcal M$-module structure on $\mathcal K$ then extends to a normal representation of $\mathcal M^{\op}$ on $ {\mathcal K_\tau}$. 

Each vector $\xi \in \mathcal K$ then gives rise to a bounded right $\mathcal M$-modular map $L_\xi: L^2(\mathcal M, \tau) \to \mathcal K_\tau$ such that $L_\xi(x) = \xi x$ for all $x \in \mathcal M \subset L^2(\mathcal M, \tau)$. To see that $L_\xi$ is bounded, just note that for $x \in \mathcal M \subset L^2(\mathcal M, \tau)$ we have 
\[
\| L_\xi(x) \|^2_\tau = \tau( \langle \xi x, \xi x \rangle_{\mathcal M} ) = \tau( x^* \langle \xi, \xi \rangle_{\mathcal M} x ) \leq \| \xi \|^2 \| x \|_2^2.
\] 
Every bounded right $\mathcal M$-modular map arises in this way, and if $\xi, \eta \in \mathcal K$, then we can recover our inner product as $\langle \xi, \eta \rangle = L_\xi^* L_\eta \in (J \mathcal M J)' \cap \mathcal B(L^2(\mathcal M, \tau)) = \mathcal M$. The mapping $L: \mathcal K \to \mathcal B(L^2(\mathcal M, \tau), \mathcal K_\tau)$ is then isometric and gives a homeomorphism between the weak$^*$-topology on $\mathcal K$ and the ultraweak topology on $ \mathcal B(L^2(\mathcal M, \tau), \mathcal K_\tau)$.

As a consequence, if $X \subset \mathcal K$ is an $\mathcal M$-invariant subset, then $X$ is weak$^*$-dense in $\mathcal K$ if and only if $X$ is dense in $\mathcal K_\tau$. Indeed, if $X$ were not dense in $\mathcal K_\tau$, and if we let $P$ denote the projection onto $X^\perp$ in $\mathcal K_\tau$, then $P$ is right $\mathcal M$-modular, and $P L_\xi = 0$ for all $\xi \in X$. If we then took any non-zero right $M$-modular map $L \in \mathcal B(L^2(\mathcal M, \tau), P \mathcal K_\tau)$, then $L = L_\eta$ for some $\eta \in \mathcal K$, and $P L_\eta \not= 0$, showing that $\eta$ is not in the weak$^*$-closure of $X$. 

Another consequence we shall use is that if $\mathcal K$ and $\mathcal H$ are two dual Hilbert $\mathcal M$-modules, $X \subset \mathcal K$ is a weak$^*$-dense $\mathcal M$-invariant subset, and $V: X \to \mathcal H$ is a right $\mathcal M$-modular map that satisfies 
\begin{equation}\label{eq:unitary}
\langle V\xi, V \eta \rangle = \langle \xi, \eta \rangle
\end{equation} 
for all $\xi, \eta \in X$, then $V$ has an extension to $\mathcal K$ that satisfies (\ref{eq:unitary}) for $\xi, \eta \in \mathcal K$ and such that $V$ is continuous with respect to the weak$^*$-topologies. Indeed, if $V_\tau$ denotes the map $V$ when viewed as a map between $\mathcal K_\tau$ and $\mathcal H_\tau$, then $V_\tau$ extends to an isometry, and we may define $V$ by $L_{V\xi} = V_\tau L_{\xi}$.

\section{Properly proximal groups}

Suppose $\Gamma$ is an infinite discrete group and we have an action by homeomorphisms on a non-empty Hausdorff topological space $X$. Recall that a pair of points $x, y \in X$ are called proximal if the orbit $\Gamma \cdot (x, y)$ has non-trivial intersection with every neighborhood of the diagonal $\Delta = \{ (x, x) \mid x \in X \} \subset X^2$. We say a pair of points $x, y \in X$ are properly proximal if for every neighborhood $O$ of $\Delta$ there is a finite set $F \subset \Gamma$ such that $(\Gamma \setminus F) \cdot (x, y) \subset O$. We say a point $x \in X$ is properly proximal if any pair of points in the orbit $\Gamma \cdot x$ are properly proximal, and we say the action $\Gamma \actson X$ is properly proximal if the set of properly proximal points in $X$ is dense.

\begin{lem}\label{lem:convex}
Let $E$ be a compact convex subset of a locally convex topological vector space $X$. Let $K\subset E$ be a compact convex subset and $U$ be relatively open in $E$ with $K\subset U$. Then there exists a convex set $V$ such that $K\subset V\subset U$ and $V$ is relatively open in $E$.
\end{lem}
\begin{proof}
For each point $x\in K$ choose an open convex neighborhood $U_x$ of zero such that $(x+U_x+U_x)\cap E\subset U$. The family $\{x+U_x : x\in K\}$ is an open cover of $K$. Therefore there exists a finite subset $F\subset K$ such that $K\subset \bigcup\{x+U_x : x\in F\}$. Put $U_0=\bigcap\{U_x : x\in F\}$.  Then $U_0$ is both open and convex. The set $K+U_0$ is convex as a sum of convex  sets, and it is open as a union of a family $\{x+U_0: x\in K\}$ of open sets. The set $V=(K+U_0)\cap E$ is relatively open in $E$, and it is convex as an intersection of two convex sets. It is clear that $K\subset V$. Now, let $x\in V\subset K+U_0$ be an arbitrary point. Then there exists a point $z\in K$ such that $x\in z+U_0$. Also there exists a point $y\in F$ such that $z\in y+U_y$. Then  $x\in y+U_y+U_0\subset  y+U_y+U_y$. Since $x\in E,$ we see that $x\in U$.
\end{proof}

\begin{lem}\label{lem:probpropproximal}
Suppose $\Gamma$ acts on a compact Hausdorff space $X$. If $\Gamma \actson X$ is properly proximal, then so is the action $\Gamma \actson {\rm Prob}(X)$.
\end{lem}
\begin{proof}
First note that the embedding of $X$ into ${\rm Prob}(X)$ as Dirac masses is a homeomorphism from $X$ into ${\rm Prob}(X)$ with the weak$^*$-topology. Thus, if $x \in X$ is a properly proximal point, then so is $\delta_{\{ x \} } \in {\rm Prob}(X)$.

We now claim that the set of properly proximal points in ${\rm Prob}(X)$ is closed under taking convex combinations. Indeed, suppose $\eta_1, \ldots, \eta_k \in {\rm Prob}(X)$ are properly proximal, and $\eta=\sum_{i=1}^kt_i\eta_i$ with $\sum_{i=1}^kt_i=1$. We view ${\rm Prob}(X) \times {\rm Prob}(X)$ as a compact convex subset of $C(X)^* \oplus_\infty C(X)^*$. Let $\mathcal O^*$ be an open neighborhood of $\Delta^*$, the diagonal of ${\rm Prob}(X) \times {\rm Prob}(X)$. In light of Lemma~\ref{lem:convex}, we may assume without loss of generality that $\mathcal O^*$ is convex. For each $i=1,\ldots,k$, there exists a finite subset $F_i\subset\Gamma$ such that $\gamma\eta_i \oplus \gamma g\eta_i\in\mathcal {O}^*$ for all $g\in \Gamma, \gamma\notin F_i$. Set $F=\cup_{i=1}^k F_i$. Then 
$\gamma\eta_i \oplus \gamma g\eta_i\in\mathcal {O}^*$ for all $g\in \Gamma, \gamma\notin F$ and for all $i=1,\ldots, k$, and we have
\begin{align*}
 \gamma \eta \oplus \gamma g \eta 
= ( \gamma \sum_{i=1}^kt_i\eta_i ) \oplus ( \gamma g \sum_{i=1}^kt_i\eta_i )
= \sum_{i=1}^k t_i (\gamma \eta_i \oplus \gamma g \eta_i )
\in \mathcal {O}^*
\end{align*}
since  $\mathcal {O}^*$ is convex.

The proof is then immediate, as the convex combination of Dirac masses is dense in ${\rm Prob}(X)$.
\end{proof}

A Banach $\Gamma$-module consists of a pair $(\pi, E)$, where $E$ is a Banach space and $\pi: \Gamma \to {\rm Isom}(E)$ is an isometric representation of $\Gamma$ on $E$. We will often drop the notation $\pi$ and by abuse of notation refer to $E$ as a Banach $\Gamma$-module. A dual Banach $\Gamma$-module consists of a dual Banach space of a Banach $\Gamma$-module, together with the natural dual representation of $\Gamma$. Note that for a dual Banach $\Gamma$-module $(\pi, E)$, the Banach $\Gamma$-module to which it is dual is part of the data, and we denote this predual of $E$ by $E_*$ so that $E = (E_*)^*$. The weak$^*$-topology on $E$ will always refer to the weak$^*$-topology with respect to this duality. 

A group $\Gamma$ is defined in \cite{BIP18} to be properly proximal if there exists an action of $\Gamma$ on a compact Hausdorff space $X$ such that there is no $\Gamma$-invariant measure on $X$ and such that ${\rm Prob}(X)$ has a properly proximal point. It will be easier for us here to consider actions on convex subsets of locally convex topological vector spaces, and so we reformulate proper proximality in this setting.

\begin{prop}\label{prop:properlyproximalequivalent}
Let $\Gamma$ be an infinite discrete group. The following are equivalent:
\begin{enumerate}[$($i$)$]
\item\label{item:1} $\Gamma$ is properly proximal.
\item\label{item:2} $\Gamma$ has a properly proximal action on a compact Hausdorff space that does not have an invariant measure. 
\item\label{item:3} There is a dual Banach $\Gamma$-module $E$ and a non-empty $\Gamma$-invariant weak$^*$-compact convex subset $K \subset E$ such that $K$ has a properly proximal point (with respect to the weak$^*$-topology), but has no fixed point. 
\item\label{item:4} There is a dual Banach $\Gamma$-module $E$ and a non-empty $\Gamma$-invariant weak$^*$-compact convex subset $K \subset E$ such that the action $\Gamma \actson K$ is properly proximal (with respect to the weak$^*$-topology) but has no fixed point. 
\item\label{item:5} $\Gamma$ has an action by affine homeomorphisms on a non-empty compact convex subset $K$ of a locally convex topological vector space such that the action $\Gamma \actson K$ is properly proximal but has no fixed point. 
\end{enumerate}
\end{prop}
\begin{proof}
(\ref{item:1}) $\implies$ (\ref{item:3}) is trivial by considering the weak$^*$-compact convex set of probability measures on a compact Hausdorff space.  (\ref{item:1}) $\implies$ (\ref{item:2}) is also trivial as we can restrict to the closure of an orbit of a properly proximal point. (\ref{item:2}) $\implies$ (\ref{item:4}) follows from Lemma~\ref{lem:probpropproximal}. (\ref{item:4}) $\implies$ (\ref{item:5}) is trivial.

Finally, both (\ref{item:3}) $\implies$ (\ref{item:1}) and (\ref{item:5}) $\implies$ (\ref{item:1}) follow from the simple observations that if $k$ is a properly proximal point in a compact Hausdorff space $K$, then $\delta_{ \{ k \} }$ is a properly proximal point in ${\rm Prob}(K)$, and if a compact convex set has an invariant measure, then the barycenter of such a measure gives a fixed point. 
\end{proof}

Given a dual Banach $\Gamma$-module $(\pi, E)$, we let $E_{\rm mix}$ denote the set of all points $x \in E$ such that we have weak$^*$-convergence $\lim_{\gamma \to \infty} \pi(\gamma)x = 0$.  Note that $E_{\rm mix}$ is a norm-closed $\Gamma$-invariant subspace of $E$, so that $(\pi, E_{\rm mix})$ is also a Banach $\Gamma$-module. 

We also have a characterization of proper proximality in terms of bounded cohomology, which is of independent interest.  A bounded 1-cocycle is a map $c: \Gamma \rightarrow E$ such that $c(\gamma_1 \gamma_2)=\gamma_1 c(\gamma_2)+c(\gamma_1)$ and $\sup_{\gamma \in \Gamma} \Vert c(\gamma) \Vert_E < \infty.$  It represents a trivial cohomology class in $H^1_b(\Gamma,E)$ if it is of the form $c(\gamma)=\gamma v - v$ for some $v \in E$.

\begin{prop}
Let $\Gamma$ be a discrete group. Then $\Gamma$ is properly proximal if and only if there is a dual Banach $\Gamma$-module $E$ such that the induced map $H^1_b(\Gamma, E_{\rm mix}) \to H^1_b(\Gamma, E)$ has non-trivial range.  
\end{prop}
\begin{proof}
Suppose $\Gamma$ is properly proximal and let $\Gamma \actson X$ be an action on a compact Hausdorff space such that ${\rm Prob}(X)$ has a properly proximal point $\eta$ but $X$ has no invariant measure. We let $E = \{ \zeta \in {\rm Meas}(X) \mid \zeta(X) = 0 \} \subset C(X)^*$ and define a bounded cocycle $c: \Gamma \to E$ by $c(\gamma) = \gamma \eta - \eta$. Since $\eta$ is properly proximal we have that the cocycle $c$ ranges in $E_{\rm mix}$.  If we had $c(\gamma) = \gamma \zeta - \zeta$ for some $\zeta \in E$ then it would follow that $\eta - \zeta \in C(X)^*$ is $\Gamma$-invariant, and as $X$ has no $\Gamma$-invariant probability measure we must then have $\eta = \zeta \in E$. However, $\eta \not\in E$ since $\eta (X) = 1$, and hence $c$ represents a non-trivial cohomology class in $H^1_b(\Gamma, E)$. 

Conversely, suppose $E$ is a dual Banach $\Gamma$-module and $c: \Gamma \to E_{\rm mix}$ is a bounded cocycle that represents a non-trivial cohomology class in $H^1_b(\Gamma, E)$. Consider the associated isometric affine action on $E$ given by $\alpha(\gamma) x = \gamma x + c(\gamma)$. Note that we have weak$^*$-topology convergence $\lim_{\gamma \to \infty} \alpha(\gamma) \alpha(g).0 -  \alpha(\gamma).0  = \lim_{\gamma \to \infty} \gamma c(g) = 0$. Thus, $0$ is a properly proximal point with respect to the action $\alpha$. 

If we let $K$ be the weak$^*$-closure of $c(\Gamma)$, then $K$ is weak$^*$-compact by the Banach-Alaoglu Theorem, and we have that $\Gamma \aactson{\alpha} K$ is properly proximal. If we had an invariant measure on $K$, then taking the barycenter would give a $\Gamma$-fixed point in $E$, which would contradict the fact that the cocycle $c$ represents a non-trivial cohomology class in $H^1_b(\Gamma, E)$. Thus $\Gamma$ is properly proximal by condition (iii) in Proposition~\ref{prop:properlyproximalequivalent}.
\end{proof}

\section{Fundamental domains for actions on von Neumann algebras}\label{section:funddomain}

\begin{defn}
Let $\Gamma \aactson{\sigma} \mathcal M$ be an action of a discrete group $\Gamma$ on a von Neumann algebra $\mathcal M$. A fundamental domain for the action is a projection $p \in \mathcal M$ so that $\{ \sigma_\gamma(p) \}_{\gamma \in \Gamma}$ gives a partition of unity. 
\end{defn}

Note that if $p \in \mathcal M$ is a fundamental domain, then we obtain an inclusion $\theta_p:\ell^\infty \Gamma \to \mathcal M$ by $\theta_p(f) = \sum_{\gamma \in \Gamma} f(\gamma) \sigma_{\gamma^{-1}}(p)$. Moreover, this embedding is equivariant with respect to the $\Gamma$-actions, where $\Gamma \actson \ell^\infty \Gamma$ is the canonical right action given by $L_\gamma(f)(x) = f( x \gamma)$. Conversely, if $\theta: \ell^\infty \Gamma \to \mathcal M$ is an equivariant embedding, then $\theta(\delta_e)$ gives a fundamental domain.

\begin{prop}\label{prop:funddomain}
Suppose $\mathcal M$ is a semi-finite von Neumann algebra with a semi-finite normal faithful trace ${\rm Tr}$, and $\Gamma \aactson{\sigma} (\mathcal M, {\rm Tr})$ is a trace-preserving action that has a fundamental domain $p$. The following are true:
\begin{enumerate}[$($i$)$]
\item\label{item:A} The map $\tau(x) = {\rm Tr}(p x p)$ is independent of the fundamental domain $p$ and defines a faithful normal semi-finite trace on $\mathcal M^\Gamma$. 
\item\label{item:C} There is a unitary operator $\mathcal F_p:  \ell^2 \Gamma \ovt L^2(\mathcal M^\Gamma, \tau) \to L^2(\mathcal M, {\rm Tr})$ that satisfies 
\[ 
\mathcal F_p( \delta_\gamma \otimes x ) = \sigma_{\gamma^{-1}}(p) x, 
\]
for $x \in \mathfrak n_\tau \subset \mathcal M^\Gamma$ and $\gamma \in \Gamma$.
\item\label{item:D} The operator $\mathcal F_p$ satisfies 
\begin{equation}\label{eq:fdunitary} 
\mathcal F_p (1 \otimes JxJ) = JxJ \mathcal F_p, \ \ \  \mathcal F_p (\rho_\gamma \otimes 1) = \sigma_\gamma^0 \mathcal F_p, \ \ \ \mathcal F_p (f \otimes 1)  = \theta_p(f) \mathcal F_p , 
\end{equation}
for $x \in \mathcal M^\Gamma$, $\gamma \in \Gamma$, and $f \in \ell^\infty \Gamma$, where $\theta_p: \ell^\infty \Gamma \to \mathcal M$ is the $\Gamma$-equivariant embedding given by $\theta_p(f) = \sum_{\gamma \in \Gamma} f(\gamma) \sigma_{\gamma^{-1}}(p)$. 
\item\label{item:H} $\mathcal F_p^* \langle \mathcal M, \mathcal M^\Gamma \rangle \mathcal F_p = \mathcal B(\ell^2 \Gamma) \ovt \mathcal M^\Gamma$. 
\item\label{item:B} We have $\mathcal M = W^*(\theta_p(\ell^\infty \Gamma), \mathcal M^\Gamma)$, and, in fact, 
\[
{\rm span} \{ \theta_p(f)x \mid f \in \ell^\infty \Gamma, x \in \mathcal M^\Gamma \}
\] 
is strong operator topology dense in $\mathcal M$.
\item\label{item:E} If $\alpha \in {\rm Aut}(\mathcal M)$ is an automorphism that preserves ${\rm Tr}$ and is $\Gamma$-equivariant, then $\alpha_{| \mathcal M^\Gamma}$ preserves $\tau$. 
\end{enumerate}
\end{prop}
\begin{proof}
If $x \in \mathcal M^\Gamma$ such that $\tau(x^*x) = 0$, then as ${\rm Tr}$ is faithful we have $xp = 0$. We then have $x \sigma_\gamma(p) = \sigma_\gamma( x p ) = 0$ for all $\gamma \in \Gamma$, and since $\sum_{\gamma \in \Gamma} \sigma_\gamma(p) = 1$ we then have $x = 0$, so that $\tau$ is faithful.

As ${\rm Tr}$ is semi-finite, there exists an increasing net of finite-trace projections $\{ q_i \}_{i \in I}$ so that $q_i \to p$ in the weak operator topology. If we set $\tilde q_i = \sum_{\gamma \in \Gamma} \sigma_\gamma( q_i)$, then as $p$ is a fundamental domain for $\Gamma$, it follows that $\{ \tilde q_i \}_{i \in I}$ gives an increasing net of projections in $\mathcal M^\Gamma$ that converges in the weak operator topology to $\sum_{\gamma \in \Gamma} \sigma_\gamma(p) = 1$, and satisfies $\tau(\tilde q_i) = {\rm Tr}(q_i) < \infty$ for each $i \in I$. Therefore $\tau$ is semi-finite. 

If $q$ is another $\Gamma$-fundamental domain then we also have 
\begin{align}
\tau(x^* x) 
&= {\rm Tr}(p x^* x p) 
= \sum_{\gamma \in \Gamma} {\rm Tr}(p x^* \sigma_\gamma(q) x p) \nonumber \\
&= \sum_{\gamma \in \Gamma} {\rm Tr}(\sigma_\gamma(q) x p x^* \sigma_\gamma(q) )
= \sum_{\gamma \in \Gamma} {\rm Tr}(q x \sigma_{\gamma^{-1}}(p) x^* q ) 
= {\rm Tr}(qx x^*q). \nonumber 
\end{align}
Thus $\tau$ is independent of the fundamental domain and defines a trace, proving (\ref{item:A}).

If $x \in \mathfrak n_\tau \subset \mathcal M^\Gamma$, then $px \in \mathfrak n_{\rm Tr}$ and we have $\| x \|_\tau^2 = \tau(xx^*) = {\rm Tr}(p x x^* p) = \| p x \|_{\rm Tr}^2$, so the map $\mathfrak n_\tau \ni x \mapsto px \in p L^2(\mathcal M, {\rm Tr})$ is isometric with respect to the trace norms. If $T \in \mathfrak n_{\rm Tr}$, then for each $\gamma \in \Gamma$ we set $a^T_\gamma = \sum_{\lambda \in \Gamma} \sigma_\lambda( p T \sigma_\gamma(p) )$. Note that since $p$ is a fundamental domain, this sum converges in the strong operator topology, and we have $a^T_\gamma \in \mathcal M^\Gamma$. We then compute
\[
pT = \sum_{\gamma \in \Gamma} p T \sigma_\gamma(p)  = \sum_{\gamma \in \Gamma} p a^T_\gamma,
\]
where the sums converge in $pL^2(\mathcal M, {\rm Tr})$. Since $T \in \mathfrak n_{\rm Tr}$ was arbitrary, this shows that $\mathfrak n_\tau \ni x \mapsto px$ has dense range in $pL^2(\mathcal M, {\rm Tr})$, showing that this map extends to a unitary from $L^2(\mathcal M^\Gamma, \tau)$ onto $pL^2(\mathcal M, {\rm Tr})$. Since $p$ is a fundamental domain we have a direct sum decomposition $L^2(\mathcal M, {\rm Tr}) = \sum_{\gamma \in \Gamma} \sigma_{\gamma^{-1}}(p)L^2(\mathcal M, {\rm Tr})$, and (\ref{item:C}) then follows easily.

Let $\gamma, g \in \Gamma$ and $x,y \in \mathfrak n_\tau.$  The following three computations verify (\ref{item:D}):
\begin{align*}
 &\mathcal F_p (1 \otimes JxJ) (\delta_g \otimes y ) = \mathcal F_p (\delta_g \otimes yx^*) =  
 \sigma_{g^{-1}}(p)yx^* = JxJ \sigma_{g^{-1}}(p)y = JxJ \mathcal F_p (\delta_g \otimes y ), 
\end{align*}
\begin{align*}
\mathcal F_p (\rho_\gamma \otimes 1)(\delta_g \otimes y )=\mathcal F_p (\delta_{g \gamma^{-1} } \otimes y ) = \sigma_{\gamma g^{-1}}(p)y=\sigma_\gamma^0(\sigma_{g^{-1}}(p)y)=\sigma_\gamma^0 \mathcal F_p (\delta_g \otimes y ), 
\end{align*}
\begin{multline*}
 \mathcal F_p (f \otimes 1) (\delta_g \otimes y ) = \mathcal F_p (f \delta_g \otimes y)=\mathcal F_p (f(g) \delta_g \otimes y)=f(g) \mathcal F_p (\delta_g \otimes y)\\=f(g)\sigma_{g^{-1}}(p)y  
=\left( \sum_{\gamma \in \Gamma} f(\gamma) \sigma_{\gamma^{-1}}(p) \right) \sigma_{g^{-1}}(p)y = \theta_p(f) \sigma_{g^{-1}}(p)y = \theta_p(f) \mathcal F_p (\delta_g \otimes y ).
\end{multline*}
Since $\langle \mathcal M, \mathcal M^\Gamma \rangle = J \mathcal M^\Gamma J'$ we see that (\ref{item:H}) follows directly from (\ref{item:D}).

As in part (\ref{item:C}), if $T \in \mathcal M$ and $\gamma_1, \gamma_2 \in \Gamma$, then 
\[
\sigma_{\gamma_1}(p)T\sigma_{\gamma_2}(p) = \sigma_{\gamma_1}(p) \left( \sum_{\gamma \in \Gamma} \sigma_\gamma( \sigma_{\gamma_1}(p) T \sigma_{\gamma_2}(p) ) \right) \in {\rm span} \{  \theta_p(f) x \mid x \in \mathcal M^\Gamma, f \in \ell^\infty \Gamma \},
\]
where the sum converges in the strong operator topology. We also have strong operator topology convergence 
\[
T = \sum_{\gamma_1, \gamma_2 \in \Gamma} \sigma_{\gamma_1}(p) T \sigma_{\gamma_2}(p),
\]
and hence (\ref{item:B}) follows. 

If $\alpha \in {\rm Aut}(\mathcal M)$ is a $\Gamma$-equivariant automorphism that preserves ${\rm Tr}$, and if $p$ is a $\Gamma$-fundamental domain, then $\alpha(p)$ is also a $\Gamma$-fundamental domain, and hence for $x \in \mathcal M^\Gamma$ we have
\[
\tau( \alpha(x^* x) ) = {\rm Tr}( p \alpha(x^*x) p ) = {\rm Tr}( \alpha(p) x^* x \alpha(p) ) = \tau(x^* x),
\] 
showing (\ref{item:E}). 
\end{proof}

\begin{prop}\label{prop:diagonal}
Suppose $\Gamma \actson (\mathcal M, {\rm Tr})$ is a trace-preserving action with fundamental domain $p$. There exists a trace-preserving isomorphism $\Delta_p: \mathcal M \rtimes \Gamma \to \mathcal B(\ell^2 \Gamma) \ovt \mathcal M^\Gamma$ such that 
\[
\Delta_p(u_\gamma) = \rho_\gamma \otimes 1, \ \ \ \Delta_p(x) = \mathcal F_p^* x \mathcal F_p
\]
for $\gamma \in \Gamma$, $x \in \mathcal M \subset \mathcal M \rtimes \Gamma$.  In particular, we have $\mathcal M \rtimes \Gamma \cong \mathcal B(\ell^2 \Gamma) \ovt \mathcal M^\Gamma \cong \langle \mathcal M, \mathcal M^\Gamma \rangle$. 
\end{prop}
\begin{proof}
We let $\mathcal F_p: \ell^2 \Gamma \ovt   L^2(\mathcal M^\Gamma, \tau)  \to L^2(\mathcal M, {\rm Tr})$ be the unitary from Proposition~\ref{prop:funddomain}. We define a unitary operator $W \in \mathcal U(\mathcal M' \ovt L\Gamma) \subset \mathcal U(L^2(\mathcal M, {\rm Tr}) \ovt \ell^2 \Gamma)$ by $W = \sum_{t \in \Gamma} J\sigma_{t} (p)J \otimes \lambda_{t}$, where $J$ is the conjugation operator for $\mathcal M \subset \mathcal B(L^2(\mathcal M, {\rm Tr}))$. 

Since $W \in \mathcal U(\mathcal M' \ovt L\Gamma)$ we see that for $x \in \mathcal M$ we have $(\mathcal F_p^* \otimes 1) W^* ( x \otimes 1) W (\mathcal F_p \otimes 1)= \mathcal F_p^* x \mathcal F_p \otimes 1$. Also, if $\gamma \in \Gamma$, then for all $r_1, s_1, r_2, s_2 \in \gamma$ and $x_1, x_2 \in \mathcal M^\Gamma$ we have
\begin{align}
\langle (\mathcal F_p^* \otimes 1) & W^* u_\gamma W (\mathcal F_p \otimes 1) (\delta_{r_1} \otimes x_1 \otimes \delta_{s_1}), (\delta_{r_2} \otimes x_2 \otimes \delta_{s_2}) \rangle \nonumber \\
& = \langle ( W^* u_\gamma W ) \sigma_{r_1^{-1}}(p) x_1 \otimes \delta_{s_1}, \sigma_{r_2^{-1}}(p) x_2 \otimes \delta_{s_2} \rangle \nonumber \\
& = \sum_{t_1, t_2 \in \Gamma} \langle u_\gamma ( \sigma_{r_1^{-1}}(p) x_1 \sigma_{t_1}(p) \otimes \delta_{t_1s_1} ), \sigma_{r_2^{-1}}(p) x_2 \sigma_{t_2}(p) \otimes \delta_{t_2 s_2} \rangle \nonumber \\
& = \sum_{t_1, t_2 \in \Gamma} \langle  \sigma_{\gamma r_1^{-1}}(p) x_1 \sigma_{\gamma t_1}(p) \otimes \delta_{\gamma t_1s_1}, \sigma_{r_2^{-1}}(p) x_2 \sigma_{t_2}(p) \otimes \delta_{t_2 s_2} \rangle. \nonumber
\end{align}

Each term in this sum of inner products is zero except when $\gamma t_1 s_1 = t_2 s_2$, $\gamma t_1 = t_2$, and $\gamma r_1^{-1} = r_2^{-1}$. Hence, this double summation is zero if $s_1 \not= s_2$ or $r_1 \gamma^{-1} \not= r_2$, and otherwise it may be simplified as 
\begin{align}
\sum_{t_1 \in \Gamma} \langle \sigma_{ r_2^{-1}}(p) x_1 \sigma_{\gamma t_1}(p), \sigma_{r_2^{-1}}(p) x_2 \sigma_{\gamma t_1}(p) \rangle
= \langle \sigma_{r_2^{-1}}(p) x_1, \sigma_{r_2^{-1}}(p) x_2 \rangle
= \langle x_1, x_2 \rangle_{L^2(\mathcal M^\Gamma, \tau)} \nonumber
\end{align}

Thus, it follows that
\begin{align}
\langle (\mathcal F_p^* \otimes 1) & W^* u_\gamma W (\mathcal F_p \otimes 1) (\delta_{r_1} \otimes x_1 \otimes \delta_{s_1}), (\delta_{r_2} \otimes x_2 \otimes \delta_{s_2}) \rangle \nonumber \\
& =
\langle (\rho_\gamma \otimes 1 \otimes 1) (\delta_{r_1} \otimes x_1 \otimes \delta_{s_1}), (\delta_{r_2} \otimes x_2 \otimes \delta_{s_2}) \rangle, \nonumber
\end{align}
and since $r_1, s_1, r_2, s_2 \in \Gamma$ and $x_1, x_2 \in \mathcal M^\Gamma$ were arbitrary, we have 
\[
(\mathcal F_p^* \otimes 1) W^* u_\gamma W (\mathcal F_p \otimes 1) = \rho_\gamma \otimes 1 \otimes 1.
\]

Hence, conjugation by $(\mathcal F_p^* \otimes 1) W^*$ defines a normal $*$-homomorphism $\Delta_p: \mathcal M \rtimes \Gamma \to \mathcal B(\ell^2 \Gamma) \ovt \mathcal M^\Gamma \otimes \mathbb C = (\mathcal F_p^* \otimes 1) (\langle \mathcal M, \mathcal M^\Gamma \rangle \otimes \mathbb C ) (\mathcal F_p \otimes 1)$, taking $u_\gamma$ to $\rho_\gamma \otimes 1 \otimes 1$ for each $\gamma \in \Gamma$, and taking $x$ to $\mathcal F_p^* x \mathcal F_p$ for each $x \in \mathcal M \subset \mathcal M \rtimes \Gamma$.

We may easily check that $\Delta_p( \sigma_{t^{-1}}( p ) )$ is the Dirac operator $\delta_t$ viewed as an operator in $\ell^\infty \Gamma \otimes \mathbb C \subset \mathcal B(\ell^2 \Gamma) \otimes \mathbb C$, hence the range of $\Delta_p$ contains both $\ell^\infty \Gamma$ and $R \Gamma$, which together generate the von Neumann algebra $\mathcal B(\ell^2 \Gamma) \subset \mathcal B(\ell^2 \Gamma) \ovt \mathcal M^\Gamma \otimes \mathbb C$. Since the range also contains $\mathcal M^\Gamma \subset \mathcal M$, it follows that the range of $\Delta_p$ equals $\mathcal B(\ell^2 \Gamma) \ovt \mathcal M^\Gamma \otimes \mathbb C$. 
\end{proof}

We note that if $x \in \mathcal M$, then we also have an explicit form for $\Delta_p(x)$. Indeed, if we view $\mathcal B(\ell^2 \Gamma) \ovt \mathcal M^\Gamma$ as $\mathcal M^\Gamma$-valued $\Gamma \times \Gamma$ matrices, then it's simple to check that $\Delta_p(x) = [x_{s, t}]_{s, t}$, where
\[
x_{s, t} = \sum_{\gamma \in \Gamma} \sigma_{\gamma}( \sigma_{t^{-1}}(p) x \sigma_{s^{-1}}(p) ) \in \mathcal M^\Gamma.
\]

\begin{prop}\label{prop:conjugatefds}
Suppose $\Gamma \actson (\mathcal M, {\rm Tr})$ is a trace-preserving action with fundamental domains $p$ and $q$. Then, using the notation above, we have $\mathcal F_q^* \mathcal F_p \in \mathcal U( L\Gamma \ovt \mathcal M^\Gamma)$ and $\Delta_p(p) ( \mathcal F_q^* \mathcal F_p ) = ( \mathcal F_q^* \mathcal F_p ) \Delta_p(q)$. 
\end{prop}
\begin{proof}
By (\ref{eq:fdunitary}) and Proposition~\ref{prop:diagonal} we have $\mathcal F_q^* \mathcal F_p \in (\rho(\Gamma) \otimes \mathbb C)' \cap \mathcal B(\ell^2\Gamma) \ovt \mathcal M^\Gamma = L\Gamma \ovt \mathcal M^\Gamma$. Moreover, by Proposition~\ref{prop:diagonal} we have 
\[
\Delta_p(p) \mathcal F_q^* \mathcal F_p 
= (\delta_e \otimes 1) \mathcal F_q^* \mathcal F_p 
= \mathcal F_q^* q \mathcal F_p
= \mathcal F_q^* \mathcal F_p \Delta_p(q).
\]
\end{proof}

\section{Von Neumann couplings}\label{sec:VNE}

\begin{defn}
Let $\Lambda$ and $\Gamma$ be countable groups. A von Neumann coupling between $\Lambda$ and $\Gamma$ consists of a semi-finite von Neumann algebra $\mathcal M$ with a faithful normal semi-finite trace ${\rm Tr}$ and a trace-preserving action $\Lambda \times \Gamma \actson \mathcal M$ such that there exist finite-trace fundamental domains $q$ and $p$ for the $\Lambda$ and $\Gamma$-actions, respectively. The index of the von Neumann coupling is the ratio ${\rm Tr}(p)/{\rm Tr}(q)$ and is denoted by $[\Gamma:\Lambda]_{\mathcal M}$. This is well-defined by Proposition~\ref{prop:funddomain}.
\end{defn}
By Definition~\ref{defn:vne}, $\Lambda$ and $\Gamma$ are von Neumann equivalent if there exists a von Neumann coupling between them. 

Note that the notion of von Neumann equivalence coincides with measure equivalence when restricting to the case when $\mathcal M$ is abelian. Also, if we have an isomorphism $\theta: L \Lambda \to L\Gamma$, then setting $\mathcal M = \mathcal B(L^2(L\Gamma))$ we have an action of $\Gamma$ by conjugation by $\rho_\gamma$, an action of $\Lambda$ by conjugation by $\theta(u_\lambda)$, and a common fundamental domain $P_{e}$, so that if $\Gamma$ and $\Lambda$ are $W^*$-equivalent, then they are also von Neumann equivalent. More generally, we have the following construction:

\begin{examp}
Suppose $\Lambda$ and $\Gamma$ are countable groups, we have trace-preserving actions $\Gamma \actson (M_1, \tau)$ and $\Lambda \actson (M_2, \tau)$, and a trace-preserving isomorphism $\theta: M_2 \rtimes \Lambda \to M_1 \rtimes \Gamma$ such that $\theta(M_1) = M_2$. Then $\theta$ extends to an isomorphism of basic constructions $\tilde \theta: \langle M_2 \rtimes \Lambda, M_2 \rangle \to \langle M_1 \rtimes \Gamma, M_1 \rangle$ such that $\tilde \theta(e_{M_2}) = e_{M_1}$. 

For $\gamma \in \Gamma$ we have $[ u_\gamma (J u_\gamma J), e_{M_1} ] = 0$ and hence $\Gamma \ni \gamma \mapsto {\rm Ad}(J u_\gamma J)$ describes a trace-preserving action of $\Gamma$ on $\langle M_1 \rtimes \Gamma, M_1 \rangle$, which pointwise fixes $M_1 \rtimes \Gamma$. In particular, we have that $\Lambda \ni \lambda \mapsto {\rm Ad}(\tilde \theta(u_\lambda))$ gives an action that commutes with the action of $\Gamma$, and we have that $e_{M_1}$ gives a fundamental domain for both the $\Gamma$ and $\Lambda$-actions. Therefore, $\langle M_1 \rtimes \Gamma, M_1 \rangle \cong M_1 \ovt \mathcal B(\ell^2 \Gamma)$ gives an index-one von Neumann coupling.
\end{examp}

Just as in the case of measure equivalence, von Neumann equivalence is an equivalence relation. Reflexivity follows by considering the trivial von Neumann $\Gamma$-coupling $\ell^\infty \Gamma$. Symmetry is obvious, and transitivity follows from the following proposition.

\begin{prop}\label{prop:vnetransitive}
Let $(\mathcal N, {\rm Tr}_{\mathcal N})$ and $(\mathcal M, {\rm Tr}_{\mathcal M})$ be $(\Sigma, \Lambda)$ and $(\Lambda, \Gamma)$ von Neumann couplings, respectively. We consider the natural action of $\Sigma, \Lambda$, and $\Gamma$ on $\mathcal N \ovt \mathcal M$, where $\Lambda$ acts diagonally. Then $\mathcal N \ovt \mathcal M$ has a $\Lambda$-fundamental domain, and the induced semi-finite trace on $(\mathcal N \ovt \mathcal M)^\Lambda$ gives a $(\Sigma, \Gamma)$ von Neumann coupling with index 
\[
[\Sigma: \Gamma]_{(\mathcal N \ovt \mathcal M)^\Lambda} = [\Sigma: \Lambda]_{\mathcal N} [\Lambda: \Gamma]_{\mathcal M}.
\]
\end{prop}
\begin{proof}
If $q$ is a $\Lambda$-fundamental domain for $\mathcal N$, then $q \otimes 1$ gives a fundamental domain for the action on $\mathcal N \ovt \mathcal M$. We therefore obtain an induced $\Sigma \times \Gamma$-invariant semi-finite normal faithful trace on $(\mathcal N \ovt \mathcal M)^\Lambda$ by Proposition~\ref{prop:funddomain}.

If $p \in \mathcal M$ is a fundamental domain for $\Gamma$, then we see that $\sum_{\lambda \in \Lambda} \sigma_\lambda(q) \otimes \sigma_\lambda(p) \in (\mathcal N \ovt \mathcal M)^\Lambda$ is a fundamental domain for $\Gamma$ with trace ${\rm Tr}(q){\rm Tr}(p) < \infty$. Similarly, if $r \in \mathcal N$ is a fundamental domain for $\Sigma$, and if $\tilde q \in \mathcal M$ is a fundamental domain for $\Lambda$, then $\sum_{\lambda \in \Lambda} \sigma_\lambda(r) \otimes \sigma_\lambda(\tilde q) \in (\mathcal N \ovt \mathcal M)^\Lambda$ is a fundamental domain for $\Sigma$ with trace ${\rm Tr}(r){\rm Tr}(\tilde q) < \infty$.

Hence, $(\mathcal N \ovt \mathcal M)^\Lambda$ is a $(\Sigma, \Gamma)$ von Neumann coupling with index
\[
[\Sigma: \Gamma]_{(\mathcal N \ovt \mathcal M)^\Lambda} 
= {\rm Tr}(q){\rm Tr}(p)/ {\rm Tr}(r) {\rm Tr}(\tilde q)
= [\Sigma: \Lambda]_{\mathcal N} [\Lambda: \Gamma]_{\mathcal M}.
\]
\end{proof}

\section{Inducing actions via semi-finite von Neumann algebras}

If $\Gamma$ is a group, then an operator $\Gamma$-module consists of a pair $(\pi, E)$, where $E$ is an operator space and $\pi: \Gamma \to CI(E)$ a homomorphism from $\Gamma$ to the group of surjective complete isometries of $E$. A dual operator $\Gamma$-module consists of a dual operator space $E = (E_*)^*$ that is an operator $\Gamma$-module such that the action of $\Gamma$ is dual to an action on $E_*$. Note that if $X = (X_*)^*$ is a dual Banach $\Gamma$-module, then we can regard $X$ also as an operator $\Gamma$-module by endowing $X_*$ with the operator space structure ${\rm min}(X_*)$, so that ${\rm max}(X) = ({\rm min}(X_*))^*$ becomes a dual operator $\Gamma$-module.

\begin{defn}\label{defn:induced}
Let $\Gamma$ and $\Lambda$ be discrete groups and suppose that $\Gamma \times \Lambda \actson (\mathcal M, {\rm Tr})$ is a trace-preserving action on a semi-finite von Neumann algebra $\mathcal M$. Let $E$ be a dual operator $\Lambda$-module.
\begin{enumerate}[$($i$)$]
\item Letting $\Gamma$ act trivially on $E$, we obtain an isometric action $\Gamma \actson \mathcal M_* \frownotimes E_*$, and hence a dual action $\Gamma \actson CB(\mathcal M_*, E) = (\mathcal M_* \frownotimes E_*)^*$, which we may then restrict to $(\mathcal M  \ovt  E)^\Lambda$. We call $(\mathcal M  \ovt  E)^\Lambda$ the dual operator $\Gamma$-module induced from $E$.

\item\label{item:part2def} If $K \subset E$ is a non-empty convex weak$^*$-closed subset that is $\Lambda$-invariant, then, considering the embedding $\mathcal M  \ovt  E \subset CB(\mathcal M_*, E)$, we let $\mathcal M  \ovt  K$ denote those maps $\Xi \in CB(\mathcal M_*, E)$ such that $\Xi(\varphi) \in K$ for each normal state $\varphi$. We then have that $\mathcal M  \ovt  K \subset \mathcal M  \ovt  E$ is a convex subset that is invariant under the actions of $\Gamma$ and $\Lambda$. Hence we have an action $\Gamma \actson ( \mathcal M  \ovt  K )^\Lambda$, which we refer to as the $\Gamma$-action induced from the $\Lambda$-action $\Lambda \actson K$. 
\end{enumerate}
\end{defn}

As motivation for Definition~\ref{defn:induced}, note that if $(X, \mu)$ is a standard measure space and $\mathcal M = L^\infty(X, \mu)$, then Proposition~\ref{prop:abelrange} gives an identification between $L^\infty(X, \mu)  \ovt  K$ and $L^\infty_{w^*}(X, \mu; K)$, so that $(L^\infty(X, \mu)  \ovt  K)^\Lambda$ can be identified as the space of $\Lambda$-equivariant measurable functions from $X$ to $K$.

\begin{lem}
Using the notation above, if $K$ is weak$^*$-compact, then $\mathcal M  \ovt  K$ is a weak$^*$-compact subset of $\mathcal M  \ovt  E$.
\end{lem}
\begin{proof}
Since $K$ is weak$^*$-compact, it is bounded, and hence $\mathcal M  \ovt  K$ is a norm bounded subset of $\mathcal M  \ovt  E$. Viewing elements in $\mathcal M  \ovt  K$ as maps from $\mathcal M_*$ to $E$, we then have that the weak$^*$-topology coincides with the topology of pointwise weak$^*$-convergence. Since $K$ is weak$^*$-closed, it follows that $\mathcal M  \ovt  K$ is also weak$^*$-closed, hence weak$^*$-compact by the Banach-Alaoglu Theorem.
\end{proof}

\begin{prop}\label{prop:fixedpoint}
Using the notation above, suppose that $\mathcal M^\Gamma$ has a normal $\Lambda$-invariant finite trace $\tau$. Then there exists a $\Gamma$-fixed point in $(\mathcal M  \ovt  K)^\Lambda$ if and only if there exists a $\Lambda$-fixed point in $K$.
\end{prop}
\begin{proof}
If $k_0 \in K$ is fixed by $\Lambda$, then we have that $1 \otimes k_0 \in (\mathcal M \ovt K)^\Lambda$ is clearly $\Gamma$-invariant. 

Conversely, suppose $\Xi \in (\mathcal M \ovt K)^\Lambda \subset CB(\mathcal M_*, K)^\Lambda$ is $\Gamma$-invariant. Under the Banach space isomorphism $CB(\mathcal M_*, E) \cong CB(E_*, \mathcal M)$, we see that we may make the identification $CB(\mathcal M_*, K)^\Gamma \cong CB( ( \mathcal M^\Gamma )_*, K)$, so that we may view $\Xi$ as a completely bounded $\Lambda$-equivariant map from $(\mathcal M^\Gamma )_*$ into $E$ taking states into $K$. Since $\Lambda$ preserves the trace $\tau$ on $\mathcal M^\Gamma$, we have that $\Xi(\tau) \in K$ is $\Lambda$-invariant.
\end{proof}

\begin{lem}\label{lem:orthsum}
Let $E = (E_*)^*$ be a dual operator space, and let $\mathcal M$ be a von Neumann algebra. Suppose $x, y \in \mathcal M$, and $\{ p_i \}_{i \in I}$ is a family of pairwise orthogonal projections in $\mathcal M$. If $\{ a_i \}_{i \in I} \subset E$ is any uniformly bounded family in $E$, then the sum $\sum_{i \in I} x p_i y \otimes a_i$ converges weak$^*$ in $\mathcal M  \ovt  E$. Moreover, we have 
\[
\| \sum_{i \in I} x p_i y \otimes a_i \| \leq \| x \| \| y \| ( \sup_{i \in I} \| a_i \| ).
\] 
\end{lem}
\begin{proof}
By representing $E$ as an ultraweakly closed subspace of a Hilbert space $\mathcal H$, it suffices to show this when $E = \mathcal B(\mathcal H)$. Since $\{ p_i \}_i$ are pairwise orthogonal, we then have that $\sum_{i \in I} p_i \otimes a_i$ converges ultraweakly and hence so does 
\[ 
\sum_{i \in I} x p_i y \otimes a_i = (x \otimes 1) \left(\sum_{i \in I}  p_i \otimes a_i \right) (y \otimes 1).
\]
Moreover, 
\[
\| \sum_{i \in I} x p_i y \otimes a_i \| \leq \| x \otimes 1 \| \| \sum_{i \in I} p_i \otimes a_i \| \| y \otimes 1 \| = \| x \| \| y \| \sup_{i \in I} \| a_i \|.
\]
\end{proof}

\begin{lem}\label{lem:embedding}
Let $E$ be a dual operator $\Lambda$-module, suppose $K \subset E$ is a non-empty $\Lambda$-invariant convex weak$^*$-closed subset, and let $\mathcal M$ be a von Neumann algebra on which $\Lambda$ acts. If the action of $\Lambda$ on $\mathcal M$ has a fundamental domain $p$ then the map $\chi_p:  K \to (\mathcal M  \ovt  K)^\Lambda$ defined by 
\[
\chi_p^k = \sum_{\lambda \in \Lambda} \sigma_\lambda(p) \otimes \lambda k
\]
gives a well-defined, weak$^*$-continuous affine isometric map. In particular, $(\mathcal M  \ovt  K)^\Lambda$ is non-empty in this case.
\end{lem}

\begin{proof}
We first note that the sum defining $\chi_p^k$ converges weak$^*$ by Lemma~\ref{lem:orthsum}. It is also easy to see by a change of variables that we have $\chi_p^k \in (\mathcal M  \ovt  E)^\Lambda$. 

If $\varphi$ is a normal state on $\mathcal M$, then we obtain a probability measure $\mu$ on $\Lambda$ given by $\mu(\lambda) = \varphi(\sigma_\lambda(p))$. Viewing $\chi_p^k$ as a map from $\mathcal M_*$ to $E$, we see that it takes $\varphi$ to the element $\int \lambda k \, d\mu(\lambda) \in K$, so that $\chi_p$ maps into $(\mathcal M  \ovt  K)^\Lambda$. 

We clearly have that $\chi_p$ is affine, and by Lemma~\ref{lem:orthsum} we have $\| \chi_p^k \| \leq \| k \|$. Also, $\| k \| = \| p \otimes k \| \leq \| \chi_p^k \|,$ so that $\chi_p$ is isometric. Also note that for each $\eta \in \mathcal M_*,$ we have $\sum_{\lambda \in \Lambda} | \eta( \sigma_\lambda(p) ) | \leq \| \eta \|.$ If $k_i \to k$ weak$^*$, it then follows that $\sum_{\lambda \in \Lambda} \eta( \sigma_\lambda(p) ) \lambda k_i \to \sum_{\lambda \in \Lambda} \eta( \sigma_\lambda(p) ) \lambda k$ weak$^*$. Since $\eta \in \mathcal M_*$ was arbitrary, this shows that $\chi_p^{k_i} \to \chi_p^k$ weak$^*$.
\end{proof}

\subsection{Properly proximal actions}

In this section we show that if $\Lambda$ has a fundamental domain and the $\Gamma$-action on $\mathcal M$ is mixing, then points that are properly proximal for a $\Lambda$-action can be induced to points that are $\Gamma$-properly proximal. At the heart of the argument is Lemma~\ref{lem:wkconvergence}, which allows us to compare the induction maps $\chi_p$ and $\chi_q$ from Lemma~\ref{lem:embedding} corresponding to different fundamental domains $p$ and $q$. 

\begin{lem}\label{lem:ssotconvergence}
Let $\mathcal M$ be von Neumann algebra, and fix $\varphi \in \mathcal M_*$. If a sequence $x_n \in \mathcal M$ converges in the ultrastrong topology to $0$, then $\lim_{n \to \infty} \| x_n \varphi \|_{*} = 0$. 
\end{lem}
\begin{proof}
By considering the polar decomposition of $\varphi$, it is enough to consider the case when $\varphi$ is a state.
Since $x_n \to 0$ in the ultrastrong topology, we have $x_n^*x_n \to 0$ in the ultraweak topology. Hence by the Cauchy-Schwarz inequality we have 
\[
\| x_n \varphi \|_* = \sup_{a \in \mathcal M, \| a \| \leq 1} | \varphi (a x_n) | \leq  \varphi(a a^*)^{1/2} \varphi(x_n^* x_n)^{1/2} \to 0.
\]
\end{proof}

\begin{lem}\label{lem:wkconvergence}
Suppose $\Lambda \actson (\mathcal M, {\rm Tr})$ is a trace-preserving action, $E$ is a dual operator $\Lambda$-module, and a point $k \in E$ is properly proximal. Fix $A \in \mathfrak m_{\rm Tr}$ and suppose $\{ \alpha_n \}_{n \in \mathbb N} \subset {\rm Aut}(\mathcal M, {\rm Tr})$ is a sequence of trace-preserving automorphisms commuting with the action of $\Lambda$ and such that $\alpha_n(A) \to 0$ in the weak operator topology. Then for any finite-trace fundamental domains $p, q \in \mathcal P(\mathcal M)$, we have weak$^*$-convergence 
\[
\lim_{n \to \infty} \chi_p^k(\alpha_n(A)) - \chi_q^k(\alpha_n(A)) = 0,
\] 
where $\chi_p^k, \chi_q^k \in (\mathcal M  \ovt  E)^\Lambda$ are defined as in Lemma~\ref{lem:embedding}.
\end{lem}
\begin{proof}
Consider the trace-preserving embedding $\Delta_p: \mathcal M \to \mathcal B(\ell^2 \Lambda) \ovt \mathcal M^\Lambda$ as given in Proposition~\ref{prop:diagonal}. This then gives a corresponding restriction map from $(\mathcal B(\ell^2 \Lambda) \ovt \mathcal M^\Lambda )_*$ to $\mathcal M_*$, and by composing this with $\chi_p^k$ (where we view $\chi_p^k \in CB(\mathcal M_*, E)$), we obtain a map (which we still denote by $\chi_p^k$) from $(\mathcal B(\ell^2 \Lambda) \ovt \mathcal M^\Lambda )_*$ into $E$. This map is $\Lambda$-equivariant, where $\Lambda$ acts on $(\mathcal B(\ell^2 \Lambda) \ovt \mathcal M^\Lambda )_*$ by conjugation with $\rho_\lambda \otimes 1$. 

Note also that the isomorphism $\mathcal M \rtimes \Lambda \cong \mathcal B(\ell^2 \Lambda) \ovt \mathcal M^\Lambda$ shows that the automorphisms $\alpha_n$ extend to trace-preserving automorphisms of $\mathcal B(\ell^2 \Lambda) \ovt \mathcal M^\Lambda$, which we also denote by $\alpha_n$, and which fix $R \Lambda \otimes \mathbb C$. Part (\ref{item:E}) of Proposition~\ref{prop:funddomain} applied to $\Lambda$ acting by conjugation on $\mathcal M \rtimes \Lambda$ then shows that $\alpha_n$ also preserves the finite trace on $L\Lambda \ovt \mathcal M^\Lambda = R\Lambda' \cap ( \mathcal B( \ell^2 \Lambda) \ovt \mathcal M^\Lambda)$.

Fix $A \in \mathfrak m_{{\rm Tr} \otimes \tau} \subset \mathcal B(\ell^2 \Lambda) \ovt \mathcal M^\Lambda$ and suppose that  $\alpha_n(A) \to 0$ weakly. Fix $t \in \Lambda$, $u \in L\Lambda \ovt \mathcal M^\Lambda$ and $v \in \mathcal M^\Lambda$. For $s \in \Lambda,$ let $P_s$ denote the rank-one projection onto $\mathbb C \delta_s \subset \ell^2 \Lambda$.  Then $\rho_s P_e \rho_s^*=P_{s^{-1}}.$ Viewing $\mathfrak m_{{\rm Tr} \otimes \tau}$ as a subset of the predual via the standard duality (see Section~\ref{sec:semifinite} above), we have 
\begin{align}
\chi_p^k( (\lambda_t \otimes v ) u \alpha_n(A) 
& -  u \alpha_n(A) (\lambda_t \otimes v ) ) \nonumber \\
&= \sum_{s \in \Lambda} ({\rm Tr} \otimes \tau)( ( (\lambda_t \otimes v ) u \alpha_n(A) - u \alpha_n(A) (\lambda_t \otimes v) ) (P_{s^{-1}} \otimes 1 ) ) s k \nonumber \\
&= \sum_{s \in \Lambda} ({\rm Tr} \otimes \tau) (  (\lambda_t \otimes v ) u \alpha_n(A)  (P_{s^{-1}} \otimes 1 ) ) (s k - s t k). \nonumber
\end{align}
Since we have weak operator topology convergence $u \alpha_n(A) \to 0$, and since $\tau$ is a finite trace on $\mathcal M^\Lambda$, it follows that for any finite set $F \subset \Lambda$ we have
\[
\sum_{s \in F} ({\rm Tr} \otimes \tau) (  (\lambda_t \otimes v ) u \alpha_n(A)  (P_{s^{-1}} \otimes 1 ) ) \to 0.
\]
Since $k$ is properly proximal, and since $\{ \alpha_n(A) \}_n$ is uniformly bounded in trace norm, it follows that we have weak$^*$-convergence 
\[
\chi_p^k( (\lambda_t \otimes x ) u \alpha_n(A) - u \alpha_n(A) (\lambda_t \otimes x ) ) \to 0.
\] 
Taking linear combinations of vectors of the form $\lambda_t \otimes v$, it follows that for all $z \in \mathbb C \Lambda \otimes_{\rm alg} \mathcal M^\Lambda$, we have weak$^*$-convergence 
\begin{equation}\label{approxcommute}
\chi_p^k( z u \alpha_n(A) - u \alpha_n(A) z )  \to 0.
\end{equation}

If $\{ z_m \}_m \subset L\Lambda \ovt \mathcal M^\Lambda$ is uniformly bounded, then $z_n$ converge to $0$ in the ultrastrong$^*$ topology if and only if  $z_n$ converge in $\| \cdot \|_2$ with respect to the trace. If this is the case, then as $\alpha_n$ are trace-preserving, we have from Lemma~\ref{lem:ssotconvergence} that 
\[
\lim_{m \to \infty} \sup_{n \in \mathbb N} \| \alpha_n^{-1}(z_m u) A \|_{{\rm Tr} \otimes \tau} = \lim_{m \to \infty} \sup_{n \in \mathbb N} \| A \alpha_n^{-1}(z_m u)  \|_{{\rm Tr} \otimes \tau}  = 0.
\] 
Kaplansky's Density Theorem then shows that we have (\ref{approxcommute}) for all $z \in L\Lambda \ovt \mathcal M^\Lambda$ and all $u \in L\Lambda \ovt \mathcal M^\Lambda$.

By Proposition~\ref{prop:conjugatefds} there exists a unitary $u \in \mathcal U(L\Lambda \ovt \mathcal M^\Lambda)$ so that $\chi_q^k(A) = \chi_p^k(u A u^*)$ for all $A \in \mathcal M_*$. We then have weak$^*$-convergence
\[
\chi_p^k(\alpha_n(A)) - \chi_q^k(\alpha_n(A))
= \chi_p^k( u^* (u \alpha_n(A)) - (u \alpha_n(A)) u^*) \to 0. 
\]
\end{proof}

\begin{prop}\label{prop:convergencept}
Suppose $\Gamma \times \Lambda \actson (\mathcal M, {\rm Tr})$ is a trace-preserving action such that the action of $\Lambda$ has a finite-trace fundamental domain and the Koopman representation $\Gamma \actson L^2(\mathcal M, {\rm Tr})$ is mixing. Suppose $E$ is a dual operator $\Lambda$-module and $K \subset E$ is a non-empty convex weak$^*$-compact $\Lambda$-invariant subset. If the action $\Lambda \actson K$ has a point that is properly proximal, then so does the induced action $\Gamma \actson (\mathcal M  \ovt  K)^\Lambda$.
\end{prop}

\begin{proof}
We fix a point $k \in K$ that is properly proximal for the action $\Lambda \actson K$. Given a finite-trace $\Lambda$-fundamental domain $p \in \mathcal M$, we let $\chi_p: K \to (\mathcal M  \ovt  K)^\Lambda$ be defined by $\chi_p^k = \sum_{s \in \Lambda} \sigma_s(p) \otimes s k$ as in Lemma~\ref{lem:embedding}, and we view $\chi_p^k$ as a $\Lambda$-equivariant map from $\mathcal M_*$ to $E$. We claim that $\chi_p^k$ is properly proximal for the induced action $\Gamma \actson (\mathcal M  \ovt  K)^\Lambda$.

Fix $g \in \Gamma$, and suppose $\{ \gamma_n \}_n \subset \Gamma$ is such that $\gamma_n \to \infty$. If $A \in \mathfrak m_{\rm Tr}$, then as the action of $\Gamma$ is mixing, we have that $\sigma_{\gamma_n}(A)$ converges to $0$ in the weak operator topology. Therefore, if we consider the $\Lambda$-fundamental domain $q = \sigma_{g}(p)$, then by Lemma~\ref{lem:wkconvergence} we have weak$^*$-convergence 
\[
\chi_p^k(\sigma_{\gamma_n^{-1}}(A)) - \chi_p^k(\sigma_{g^{-1} \gamma_n^{-1}}(A) ) = \chi_p^k(\sigma_{\gamma_n^{-1}}(A)) - \chi_q^k(\sigma_{\gamma_n^{-1}}(A)) \to 0.
\]  
As the set of such $A$ is dense in $\mathcal M_*$, the result follows.
\end{proof}

\begin{thm}\label{thm:properlyprox}
Suppose $\Gamma \times \Lambda \actson (\mathcal M, {\rm Tr})$ is a trace-preserving action such that $\mathcal M^\Gamma$ has a normal $\Lambda$-invariant finite trace, the action of $\Lambda$ on $\mathcal M$ has a finite-trace fundamental domain,  and the Koopman representation $\Gamma \actson L^2(\mathcal M, {\rm Tr})$ is mixing. If $\Lambda$ is properly proximal, then so is $\Gamma$. 
\end{thm}
\begin{proof}
This follows from Propositions~\ref{prop:properlyproximalequivalent}, \ref{prop:fixedpoint} and \ref{prop:convergencept}.
\end{proof}

\begin{proof}[Proof of Theorem~\ref{thm:vnepropprox}]
From Proposition~\ref{prop:funddomain}, the existence of a fundamental domain for $\Gamma$ implies that the Koopman representation is a multiple of the left-regular representation, and hence is mixing for any infinite group. The result then follows from Theorem~\ref{thm:properlyprox}.
\end{proof}

\subsection{Inducing unitary representations}

Suppose $\Lambda \actson^{\sigma}  (\mathcal M, {\rm Tr})$ is a trace-preserving action on a semi-finite von Neumann algebra and $\pi: \Lambda \to \mathcal U(\mathcal H)$ is a unitary representation. In Section~\ref{sec:HilbertC} we gave a dual Hilbert $\mathcal M$-module structure to $\mathcal M  \ovt  \mathcal H$ that satisfies
\[
\langle a \otimes \xi, b \otimes \eta \rangle_{\mathcal M} = \langle \eta, \xi \rangle a^*b
\]
for all $a, b \in \mathcal M$ and $\xi, \eta \in \mathcal H$. Thus for $s \in \Lambda$ and $x, y \in \mathcal M  \ovt  \mathcal H$, we have
\begin{equation}\label{eq:unitinduce}
\langle ( \sigma_s \otimes \pi(s) ) x, ( \sigma_s \otimes \pi(s) ) y \rangle_{\mathcal M} = \sigma_s( \langle x, y \rangle_{\mathcal M} ).
\end{equation}
The space of fixed points $(\mathcal M  \ovt  \mathcal H)^\Lambda$ then becomes a dual Hilbert $\mathcal M^\Lambda$-module.

Left multiplication of $\mathcal M$ on itself gives a normal representation of $\mathcal M$ on $\mathcal M  \ovt  \mathcal H$. Thus the space of fixed points $(\mathcal M  \ovt  \mathcal H)^\Lambda$ is endowed with a normal $\mathcal M^\Lambda$-representation. 

If $\tau$ is a faithful normal trace on $\mathcal M^\Lambda$, we obtain a positive definite scalar-valued inner product on $(\mathcal M  \ovt  \mathcal H)^\Lambda$ by $\langle y, x \rangle = \tau( \langle x, y \rangle_\mathcal M )$. We denote the corresponding Hilbert space completion as $(\mathcal M  \ovt  \mathcal H)^\Lambda_\tau$, 
which we then see is an $\mathcal M^\Lambda$-correspondence in the sense of Connes \cite[Chapter 5, Appendix B]{Co94}. 

If we are also given a trace-preserving action $\Gamma \aactson{\sigma} (\mathcal M, {\rm Tr})$ that commutes with the $\Lambda$-action, then we see that (\ref{eq:unitinduce}) also holds for the $\Gamma$-action. Hence if $\Gamma$ preserves the trace $\tau$ on $\mathcal M^\Lambda$, we obtain a unitary representation $\Gamma \actson (\mathcal M  \ovt  \mathcal H)^\Lambda_\tau$.

\begin{defn}\label{defn:inducedrep}
Suppose $\pi: \Lambda \to \mathcal U(\mathcal H)$ is a unitary representation and $\Gamma \times \Lambda \actson (\mathcal M, {\rm Tr})$ is a trace-preserving action on a semi-finite von Neumann algebra such that the action of $\Lambda$ admits a finite-trace fundamental domain. We let $\tau$ denote the $\Gamma$-invariant trace on $\mathcal M^\Lambda$ given by Proposition~\ref{prop:funddomain}. We say that the representation $\Gamma \actson (\mathcal M  \ovt  \mathcal H)^\Lambda_\tau$ is induced from $\pi$, and we denote this representation by $\pi_{\mathcal M}$.         

As an $\mathcal M^\Lambda$-correspondence, we say that $(\mathcal M  \ovt  \mathcal H)^\Lambda_\tau$ is the correspondence induced from $\pi$.
\end{defn}

\begin{prop}\label{prop:cocylerep}
Suppose $\pi: \Lambda \to \mathcal U(\mathcal H)$ is a unitary representation and $\Lambda \actson^\sigma (\mathcal M, {\rm Tr})$ is a trace-preserving action on a semi-finite von Neumann algebra that has a finite-trace fundamental domain $p$. There exists an isomorphism of dual Hilbert $\mathcal M^\Lambda$-modules $V_p: \mathcal M^\Lambda  \ovt  \mathcal H \to (\mathcal M  \ovt  \mathcal H)^\Lambda$ such that 
\[
V_p( x \otimes \xi ) = \sum_{t \in \Lambda} \sigma_t(p) x \otimes \pi(t) \xi
\] 
for all $x \in \mathcal M^\Lambda$ and $\xi \in \mathcal H$. 
\end{prop}
\begin{proof}
Note first that by Lemma~\ref{lem:orthsum}, when restricted to the algebraic tensor product, the map $V_p: \mathcal M^\Lambda \otimes \mathcal H \to (\mathcal M \ovt \mathcal H)^\Lambda$ is well-defined. Moreover, for $x, y \in \mathcal M$ and $\xi, \eta \in \mathcal H$, we have
\begin{align}
\langle V_p(x \otimes \xi), V_p(y \otimes \eta) \rangle_{\mathcal M} 
&= \sum_{s, t \in \Lambda} \langle \pi(s) \eta, \pi(t) \xi \rangle x^* \sigma_t(p) \sigma_s(p) y \nonumber \\
&= \langle \eta, \xi \rangle x^*y \nonumber \\
&= \langle x \otimes \xi, y \otimes \eta \rangle_{\mathcal M^\Lambda}. \nonumber
\end{align}
As described in Section~\ref{sec:HilbertC}, it follows that $V_p$ has a weak$^*$-continuous extension $V_p: \mathcal M^\Lambda \ovt \mathcal H \to (\mathcal M \ovt \mathcal H)^\Lambda$ that preserves the inner product; and to see that $V$ is surjective, it suffices to show that the range of $V_p$ is dense when viewed as a map into $(\mathcal M \ovt \mathcal H)^\Lambda_\tau$, where $\tau$ is the trace given by Proposition~\ref{prop:funddomain}, i.e., $\tau(x) = {\rm Tr}(pxp)$ for $x \in \mathcal M^\Lambda$. 

Suppose therefore that we have $\zeta_0 \in ( \mathcal M \ovt \mathcal H)^\Lambda_\tau$, orthogonal to the range of $V$.  

If $s \in \Lambda$, $x \in \mathcal M^\Lambda$, $\zeta \in (\mathcal M \ovt \mathcal H )^\Lambda$, and $\xi \in \mathcal H$, then using (\ref{eq:unitinduce}) we have 
\begin{align}
\langle \sigma_s(p) x  \otimes \xi, \sigma_s(p) \zeta \rangle_{\rm Tr}
&= {\rm Tr}( \langle \zeta, \sigma_s(p) x \otimes \xi \rangle_{\mathcal M}) \nonumber \\ 
&= \sum_{t \in \Lambda} {\rm Tr}(p \sigma_t( \langle \zeta, \sigma_s(p) x \otimes \xi \rangle_{\mathcal M} ) ) \nonumber \\
&=  {\rm Tr}(p \langle \zeta, \sum_{t \in \Lambda} \sigma_{ts}(p) x \otimes \pi(t) \xi \rangle_{\mathcal M} ) \nonumber \\
&= \tau( \langle \zeta, V_p( x \otimes \pi(s^{-1}) \xi )\rangle_{\mathcal M^\Lambda} ) \nonumber \\
&= \langle V_p( x \otimes \pi(s^{-1}) \xi ), \zeta \rangle_\tau. \nonumber
\end{align}
Approximating $\zeta_0$ by elements in $( \mathcal M \ovt \mathcal H)^\Lambda$ it follows that
\[
\langle  \sigma_s(p) x \otimes \xi, \sigma_s(p) \zeta_0 \rangle_{\rm Tr} = 0.
\]

By part (\ref{item:B}) of Proposition~\ref{prop:funddomain} we have that ${\rm span} \{ \sigma_s(p)x \otimes \xi \mid x \in \mathcal M^\Lambda, \xi \in \mathcal H \}$ is weak$^*$-dense in $\sigma_s(p)\mathcal M \ovt \mathcal H$, and hence it follows that $\sigma_s(p)\zeta_0 = 0$. Since $s \in \Lambda$ is arbitrary, it then follows that $\zeta_0 = 0$.
\end{proof}

A motivating example is when $\mathcal M = \mathcal B(\ell^2 \Lambda)$ and the action $\sigma: \Lambda \to {\rm Aut}( \mathcal B(\ell^2 \Lambda) )$ is given by $\sigma_t(T) = \rho_t T \rho_t^*$, where $\rho: \Lambda \to \mathcal U(\ell^2 \Lambda)$ is the right-regular representation. Then $\mathcal M^\Lambda = L\Lambda$ and the above process describes a method of inducing representations of $\Lambda$ to normal Hilbert $L\Lambda$-bimodules.

There is another, extensively used, method of inducing representations to normal Hilbert bimodules, which was originally discovered by Connes (see \cite{Co82, Ch83, CoJo85, Po86}). Given a unitary representation $\pi: \Lambda \to \mathcal U(\mathcal H)$, set $\mathcal K = \ell^2 \Lambda  \ovt \mathcal H $, and consider the representations $\lambda \otimes \pi$, and $1 \otimes \rho$ of $\Lambda$ in $\mathcal U(\mathcal K)$. The Fell unitary $U: \mathcal K \to \mathcal K$ given by $U(\delta_t \otimes \xi) = \delta_t \otimes \pi(t) \xi$ satisfies $U ( \lambda \otimes \pi ) U^* = \lambda \otimes 1$, and thus both representations $\lambda \otimes \pi$ and $\rho \otimes 1$ extend to give commuting normal representations of $L\Lambda$ and $L\Lambda^{\rm op}$ in $\mathcal B(\mathcal K)$.

The following proposition shows that, for this example, the induced bimodule described in Definition~\ref{defn:inducedrep} is isomorphic to Connes' induced bimodule.

\begin{prop}\label{prop:isomorphicbimodules}
Let $\Lambda$ be a discrete group, and $\pi: \Lambda \to \mathcal U(\mathcal H)$ a unitary representation. Then there exists a unitary $V: \ell^2 \Lambda \ovt \mathcal H  \to (\mathcal B(\ell^2 \Lambda)  \ovt  \mathcal H)^\Lambda_\tau$ that induces an isomorphism of $L\Lambda$-bimodules.
\end{prop}
\begin{proof}
For $r \in \Lambda$ we let $p_r$ be the rank-one projection onto $\mathbb C \delta_r \subset \ell^2 \Lambda$. We let $V_{p_e}: L\Lambda \ovt \mathcal H \to (\mathcal B(\ell^2 \Lambda) \ovt \mathcal H)^\Lambda$ be as in Proposition~\ref{prop:cocylerep}. If $s, t \in \Lambda$ and $\xi \in \mathcal H$, then 
\begin{align}
(\lambda_s \otimes 1) V_{p_e} (\delta_t \otimes \xi)
&=  (\lambda_s \otimes 1) \sum_{r \in \Lambda} p_r \lambda_t \otimes \pi(r^{-1}) \xi \nonumber \\
&= \sum_{r \in \Lambda} p_{sr} \lambda_{st} \otimes \pi(r^{-1}) \xi \nonumber \\
&= V_{p_e} (\lambda_s \otimes \pi(s)) (\delta_t \otimes \xi). \nonumber
\end{align}

Viewing $L\Lambda$ as a dense subspace of $\ell^2 \Lambda$ and taking completions shows that $V_{p_e}$ extends to a unitary $V_{p_e}: \ell^2 \Lambda \ovt \mathcal H \to (\mathcal B(\ell^2 \Lambda)  \ovt  \mathcal H)^\Lambda_\tau$ that intertwines the $L\Lambda$-module structures defined above.  As $V_{p_e}$ is also right $L\Lambda$-modular, the result  follows easily.
\end{proof}

\begin{lem}\label{lem:wkcontain}
Suppose $\pi: \Lambda \to \mathcal U(\mathcal H)$ and $\rho: \Lambda \to \mathcal U(\mathcal K)$ are unitary representations and $\Lambda \actson^\sigma (\mathcal M, {\rm Tr})$ is a trace-preserving action on a semi-finite von Neumann algebra. For finite-trace fundamental domains $p, q \in \mathcal M$, let $V_p$ and $V_q$ respectively be the maps defined in Proposition~\ref{prop:cocylerep}. Suppose $G$ is a finite index set, $I$ is an infinite index set, and we have functions $\xi: G \to \mathcal K$ and $\xi_i: G \to \mathcal H$  such that $\sup_{i \in I, k \in G} \| \xi_i^k \| < \infty$, and for all $t \in \Lambda$ and $k, \ell \in G$ we have 
\[
\langle \pi(t) \xi_i^k, \xi_i^\ell \rangle \to \langle \rho(t) \xi^k, \xi^\ell \rangle.
\] 
Then for all $x, y \in \mathcal M^\Lambda$ and for all $k, \ell \in G$, we have 
\[
\langle V_p (x \otimes \xi_i^k), V_q (y \otimes \xi_i^\ell) \rangle_\tau \to \langle V_p (x \otimes \xi^k), V_q(y \otimes \xi^\ell) \rangle_\tau.
\] 
\end{lem}
\begin{proof}
We compute 
\begin{align}
\langle V_q (y \otimes \xi_i^\ell), V_p (x \otimes \xi_i^k) \rangle_{\mathcal M}
&= \sum_{s, t \in \Lambda} \langle \sigma_s(q) y \otimes \pi(s) \xi_i^\ell, \sigma_t(p) x \otimes \pi(t) \xi_i^k \rangle_{\mathcal M} \nonumber \\
&= \sum_{s, t \in \Lambda} y^* \sigma_s(q)\sigma_t(p) x \langle \pi(t)\xi_i^k, \pi(s) \xi_i^\ell \rangle \nonumber \\
&= \sum_{s, t \in \Lambda} y^*  \sigma_t( \sigma_s(q)p ) x \langle \xi_i^k, \pi(s) \xi_i^\ell \rangle. \label{eq:ipsum} 
\end{align}

Hence, 
\begin{align}
\langle V_p (x \otimes \xi_i^k), V_q (y \otimes \xi_i^\ell) \rangle_\tau
&= \sum_{s, t \in \Lambda}  {\rm Tr}(p y^* \sigma_t( \sigma_s(q)p ) x) \langle \xi_i^k, \pi(s) \xi_i^\ell \rangle \nonumber \\
&= {\rm Tr}\left( y^* \left( \sum_{s \in \Lambda} \sigma_s(q) \langle \xi_i^k, \pi(s) \xi_i^\ell \rangle \right) p x \right). \nonumber
\end{align}

By normality of the trace, given $\varepsilon > 0$ there exists $F \subset \Lambda$ finite so that for all $f \in \ell^\infty \Lambda$ we have
\[
\left| {\rm Tr}\left( y^* \left( \sum_{s \not\in F} \sigma_s(q) f(s) \right) px \right) \right|
< \varepsilon \| f \|_\infty.
\]

Thus, using Cauchy-Schwarz, for all $k, \ell \in G$, 
\begin{align}
\limsup_{i \to \infty} & | \langle V_p (x \otimes \xi_i^k),  V_q (y \otimes \xi_i^\ell) \rangle_\tau - \langle V_p (x \otimes \xi^k), V_q(y \otimes \xi^\ell) \rangle_\tau |  \nonumber \\
&\leq 2\varepsilon \sup_i \| \xi_i^k \| \| \xi_i^\ell \|  + \limsup_{i \to \infty} \sum_{s \in F} \tau \left( x^* \left( \sum_{t \in \Lambda} \sigma_t( \sigma_s(p)q ) \right) y \right)  \nonumber \\
&\qquad {} \qquad {} \qquad {} \qquad {} \qquad {} \qquad {} \cdot | \langle \xi_i^\ell, \pi(s) \xi_i^k \rangle - \langle \xi^\ell, \rho(s) \xi^k \rangle | \nonumber \\
&= 2 \varepsilon \sup_i \| \xi_i^k \| \| \xi_i^\ell \|.  \nonumber
\end{align}
As $\varepsilon > 0$ was arbitrary, the result follows.
\end{proof}

\begin{lem}\label{lem:mix}
Suppose $\pi: \Lambda \to \mathcal U(\mathcal H)$ is a mixing representation and $\Lambda \actson^\sigma (\mathcal M, {\rm Tr})$ is a trace-preserving action on a semi-finite von Neumann algebra. Suppose we have finite-trace $\Lambda$-fundamental domains $p_i \in \mathcal M$ such that $p_i \to 0$ in the weak operator topology. Then for any $\Lambda$-fundamental domain $p$ and $\xi, \eta \in \mathcal H$, we have
\[
 \lim_{i \to \infty} \sup_{x, y \in (\mathcal M^\Lambda)_1} | \langle V_p (x \otimes \xi), V_{p_i} (y \otimes \eta) \rangle_\tau | = 0.
\]
\end{lem}
\begin{proof}
Fix $p \in \mathcal M$ a finite-trace fundamental domain and $\xi, \eta \in \mathcal H$. Then for $x, y \in \mathcal M^\Lambda$ we may compute as in (\ref{eq:ipsum})
\[
\langle V_p (x \otimes \xi), V_{p_i} (y \otimes \eta) \rangle_\tau
= \sum_{s \in \Lambda} x^* \left( \sum_{t \in \Lambda} \sigma_t( \sigma_s(p) p_i ) \right) y \langle \eta, \pi(s) \xi \rangle. 
\]

Fix $\varepsilon > 0$. Since $\pi$ is a mixing representation, there exists $F \subset \Lambda$ finite so that $| \langle \eta, \pi(s) \xi \rangle | < \varepsilon$ for all $s \not\in F$. As $p_i \to 0$ weakly, we have 
\[
\lim_{i \to \infty} \sum_{s \in F} \left| \tau \left( \sum_{t \in \Lambda} \sigma_t(\sigma_s(p)p_i)  \right) \right| = 0.
\]
Hence 
\begin{align}
 \limsup_{i \to \infty} & \sup_{x, y \in (\mathcal M^\Lambda)_1} | \langle V_p (x \otimes \xi), V_{p_i} (y \otimes \eta) \rangle_\tau | \nonumber \\
&\leq  \limsup_{i \to \infty}  \sup_{x, y \in (\mathcal M^\Lambda)_1}  \sum_{s \not\in F} \left| \tau \left( x^* \left( \sum_{t \in \Lambda} \sigma_t( \sigma_s(p) p_i ) \right) y \right) \langle \eta, \pi(s) \xi \rangle \right| < \varepsilon. \nonumber 
\end{align}
Since $\varepsilon > 0$ was arbitrary, and since vectors of the form $x \otimes \xi$ span a weak$^*$-dense subset of $\mathcal M^\Lambda \ovt \mathcal H$, the result follows. 
\end{proof}

The following proposition generalizes results in Section 8 from \cite{Fu99B}. In the case when $\mathcal M$ is associated to a $W^*$-equivalence as in Proposition~\ref{prop:isomorphicbimodules}, this follows from results in \cite{Ch83, CoJo85}.

\begin{prop}\label{prop:inducedunitary}
Suppose $\pi: \Lambda \to U(\mathcal H)$ and $\rho: \Lambda \to U(\mathcal K)$ are two unitary representations of $\Lambda$, and $\Gamma \times \Lambda \actson (\mathcal M, {\rm Tr})$ is a von Neumann coupling. The following hold:
\begin{enumerate}[$($i$)$]
\item\label{item:induceA} If $\pi \prec \rho$, then $\pi_{\mathcal M} \prec \rho_{\mathcal M}$. 
\item\label{item:induceB} If $\pi$ is mixing, then $\pi_{\mathcal M}$ is mixing. 
\item\label{item:induceC} $\lambda_\mathcal M$ is a multiple of the left-regular representation of $\Gamma$.
\item\label{item:induceD} If $\pi$ is weak mixing, then $\pi_{\mathcal M}$ has no non-zero invariant vectors.
\end{enumerate}
\end{prop}
\begin{proof}
Suppose first that $\pi \prec \rho$. Replacing $\rho$ with $\rho^{\oplus \infty}$, we may assume that $\rho$ has infinite multiplicity. Fix $G$ a finite set, and suppose $\xi: G \to \mathcal K$ is a map. Since $\pi \prec \rho$, there exists a net $\xi_i: G \to \mathcal H$ such that for all $t \in \Lambda$, we have $\langle \pi(t) \xi_i^k, \xi_i^\ell \rangle \to \langle \rho(t) \xi^k, \xi^\ell \rangle$. By Lemma~\ref{lem:wkcontain}, for all $x, y \in \mathcal M^\Lambda$ and $\gamma \in \Gamma$, we then have 
\begin{align}
\langle (\sigma_\gamma \otimes 1 )V_{p} ( x \otimes \xi_i^k ), V_p( y \otimes \xi_i^\ell ) \rangle_\tau
&= \langle V_{\sigma_\gamma(p)} ( \sigma_\gamma(x) \otimes \xi_i^k ), V_p( y \otimes \xi_i^\ell ) \rangle_\tau \nonumber \\
&\to \langle V_{\sigma_\gamma(p)} ( \sigma_\gamma(x) \otimes \xi^k), V_p( y \otimes \xi^\ell )  \rangle_\tau \nonumber \\
&= \langle (\sigma_\gamma \otimes 1 )V_{p} ( x \otimes \xi^k), V_p( y \otimes \xi^\ell )  \rangle_\tau. \nonumber
\end{align}
As elements of the form $x \otimes \xi$ span a dense subset of $( \mathcal M^\Lambda \ovt \mathcal H )_\tau$ this then shows (\ref{item:induceA}).

If $\pi$ is mixing and $\gamma \to \infty$, then for a fixed $\Lambda$-fundamental domain $p \in \mathcal M$, we have that $\sigma_\gamma(p) \to 0$ weakly. Hence Lemma~\ref{lem:mix} shows that for all $\xi, \eta \in \mathcal H$ and $x, y \in \mathcal M^\Lambda$, we have
\[
\lim_{\gamma \to \infty} \langle (\sigma_\gamma \otimes 1) V_p (x \otimes \xi), V_p( y \otimes \eta) \rangle_\tau
= \lim_{\gamma \to \infty} \langle V_{\sigma_\gamma(p)} \sigma_\gamma(x) \otimes \xi, V_p (y \otimes \eta \rangle_\tau
=0.
\]
Thus $\pi_{\mathcal M}$ is also mixing, which then shows (\ref{item:induceB}).

We define the map $\mathcal F: (\mathcal M \ovt \ell^2 \Lambda)^\Lambda \to \mathcal M$ by $\mathcal F(\xi) = \langle 1 \otimes \delta_e, \xi \rangle_{\mathcal M}$. For $x \in \mathcal M^\Lambda$ and $t \in \Lambda$, we then have $\mathcal F( V_p( x \otimes \delta_t))  = \langle 1 \otimes \delta_e, V_p( x \otimes \delta_t) \rangle = \sigma_{t^{-1}}(p) x$. Hence if we also have $y \in \mathcal M^\Lambda$ and $s \in \Lambda$, then
\begin{align*}
\langle \mathcal F(V_p( x \otimes \delta_t)), \mathcal F(V_p(y \otimes \delta_s)) \rangle_{\rm Tr}
&= \delta_{s, t} {\rm Tr}( x^* \sigma_{t^{-1}}(p) y) \\
&= \delta_{s, t} \tau( \sigma_t(y x^*) ) \\
&= \delta_{s, t} \tau( x^* y) \\
&= \langle x \otimes \delta_t, y \otimes \delta_s \rangle_\tau \\
&= \langle V_p( x \otimes \delta_t), V_p(y \otimes \delta_s) \rangle_\tau.
\end{align*}
Thus, $\mathcal F$ extends to an isometry $\mathcal F: (\mathcal M \ovt \ell^2 \Lambda)^\Lambda_{\tau} \to L^2(\mathcal M, {\rm Tr})$. Moreover, by part (\ref{item:B}) of Proposition~\ref{prop:funddomain} we see that ${\rm span} \{ \mathcal F( V_p(x \otimes \delta_t) ) \mid x \in \mathcal M^\Lambda, t \in \Lambda \}$ is dense in $L^2(M, {\rm Tr})$, hence $\mathcal F$ is unitary. 

As $\mathcal F$ commutes with the action of $\Gamma$, we then see that $\mathcal F$ implements an intertwiner between the representation $\lambda_{\mathcal M}$ and the Koopman representation on $L^2(\mathcal M, {\rm Tr})$. Since $\Gamma$ has a finite-measure fundamental domain, the latter representation is isomorphic to an amplification of the left regular representation by part (\ref{item:C}) of Proposition~\ref{prop:funddomain}. This then establishes (\ref{item:induceC}).

We now suppose that $\pi_{\mathcal M}$ has a non-zero invariant vector in $(\mathcal M \ovt \mathcal H)^\Lambda_\tau$. First, note that this then implies that there is a non-zero $\Gamma$-invariant vector in $(\mathcal M \ovt \mathcal H)^\Lambda$. Indeed, if $\xi \in (\mathcal M \ovt \mathcal H)^\Lambda_\tau$ is a $\Gamma$-invariant vector, then we may approximate $\xi$ by some $\eta \in (\mathcal M \ovt \mathcal H)^\Lambda$ so that $\| \xi - \eta \|_\tau < \frac{1}{2} \| \xi \|$. If we let $\xi_0$ be the unique element of minimal $\| \cdot \|_\tau$ in the $\| \cdot \|_\tau$-closed convex closure hull of $\{ \pi_{\mathcal M}(\gamma) \eta \mid \gamma \in \Gamma \}$, then $\xi_0$ is also $\Gamma$-invariant, and we have $\| \xi_0 - \xi \| \leq \| \eta - \xi \| < \frac{1}{2} \| \xi \|$, so that $\xi_0$ is non-zero. Closed balls in $\mathcal M \ovt \mathcal H$ are weak$^*$-compact by the Banach-Alaoglu theorem and hence we see that $\xi_0 \in (\mathcal M \ovt \mathcal H)^\Lambda \subset (\mathcal M \ovt \mathcal H)^\Lambda_\tau$. 

We therefore have a non-zero vector in $(\mathcal M \ovt \mathcal H)^{\Gamma \times \Lambda} \cong (\mathcal M^\Gamma \ovt \mathcal H )^\Lambda$. Recall that we endow $\mathcal H$ with its column operator space structure coming from the isomorphism $\mathcal H \cong \mathcal B( \mathbb C, \mathcal H)$. We therefore consider $\xi_0 \in (\mathcal M^\Gamma \ovt \mathcal B(\mathbb C, \mathcal H) )^\Lambda$, and we then obtain a non-zero positive operator $| \xi_0 | \in (\mathcal M^\Gamma \ovt HS(\mathcal H) )^\Lambda$, where $HS(\mathcal H)$ denotes the space of Hilbert-Schmidt operators on $\mathcal H$. As $\tau_\Gamma \otimes {\rm Tr}$ gives a faithful trace on $\mathcal M^\Gamma \ovt \mathcal B(\mathcal H)$, we then obtain a non-zero $\Lambda$-invariant vector $(\tau_\Gamma \otimes {\rm id})( | \xi_0 | ) \in HS(\mathcal H)$. This then shows that $\pi$ is not weak mixing, establishing (\ref{item:induceD}).
\end{proof}

\begin{proof}[Proof of Theorem~\ref{thm:vneamenaT}]
Amenability is characterized by having the left regular representation weakly contain the trivial representation, thus (\ref{item:induceA}) and (\ref{item:induceC}) in Proposition~\ref{prop:inducedunitary} show that amenability is preserved under von Neumann equivalence. 

Similarly, the Haagerup property is characterized by having a mixing representation that weakly contains the trivial representation. Thus, (\ref{item:induceA}) and (\ref{item:induceB}) in Proposition~\ref{prop:inducedunitary} show that the Haagerup property is preserved under von Neumann equivalence. 

Finally, if $\Gamma$ has property (T) and $\pi$ is a representation of $\Lambda$ that weakly contains the trivial representation, then since $1_{\mathcal M}$ contains the trivial representation for $\Gamma$, it follows that $\pi_{\mathcal M}$ also weakly contains the trivial representation. Property (T) then implies that $\pi_{\mathcal M}$ contains non-zero $\Gamma$-invariant vectors, and by (\ref{item:induceD}) in Proposition~\ref{prop:inducedunitary} it follows that $\pi$ is not weak mixing. It then follows from \cite[Theorem 1]{BeVa93} that $\Lambda$ also has property (T). 
\end{proof}

\section{Von Neumann equivalence for finite von Neumann algebras}\label{sec:vnefinite}

\begin{defn}\label{defn:funddomvn}
Let $M \subset \mathcal M$ be an inclusion of semi-finite von Neumann algebras. A fundamental domain for $M$ inside of $\mathcal M$ consists of a realization of  the standard representation $M \subset \mathcal B(L^2(M))$ as an intermediate von Neumann subalgebra $M \subset \mathcal B(L^2(M)) \subset \mathcal M$. The fundamental domain is finite if finite-rank projections in $\mathcal B(L^2(M))$ are mapped to finite projections in $\mathcal M$. 
\end{defn} 

Note that if $P = \mathcal B(L^2(M) )' \cap \mathcal M$, then we have an isomorphism 
\[
\mathcal M \cong \mathcal B(L^2(M)) \ovt P
\] 
where $M$ acts standardly as $M \otimes \mathbb C$. 

\begin{lem}\label{lem:conjfd}
Let $M \subset \mathcal M$ be an inclusion of von Neumann algebras with $\mathcal M$ being semi-finite and $M$ being finite and $\sigma$-finite. Then any two fundamental domains for $M$ are conjugate by a unitary in $M' \cap \mathcal M$. 
\end{lem}
\begin{proof}
Let $\mathcal C_{\rm Tr}$ be a faithful normal semi-finite center-valued trace on $\mathcal M$, and let $\tau$ be a faithful normal trace on $M$. Let  $\mathcal F$ denote the collection of finite dimensional subspaces of $M \subset L^2(M, \tau)$, which we order by inclusion. Then the net $\{ P_{V} \}_{V \in \mathcal F} \subset \mathcal B(L^2(M, \tau))$ converges to the identity in the strong operator topology. If $\theta_1, \theta_2: \mathcal B(L^2(M, \tau)) \to \mathcal M$ are embeddings that restrict to the identity on $M$, and if $V \in \mathcal F$ has an orthonormal basis $\{ x_1, \ldots, x_k \} \subset M$, then we have
\begin{align}
\mathcal C_{\rm Tr}( \theta_1(P_{\mathbb C \hat 1}) \theta_2(P_V) )
&= \sum_{i = 1}^k \mathcal C_{\rm Tr}(  \theta_1(P_{\mathbb C \hat 1}) \theta_2( x_j P_{\mathbb C \hat 1} x_j^*))  \nonumber \\
&= \sum_{i = 1}^k \mathcal C_{\rm Tr}( \theta_1( x_j^* P_{\mathbb C \hat 1} x_j ) \theta_2( P_{\mathbb C \hat 1} ) ) \nonumber \\
&= \mathcal C_{\rm Tr}( \theta_1( P_{V^*} ) \theta_2( P_{\mathbb C \hat 1}) ). \nonumber
\end{align}
Taking the limit over $\mathcal F$, we see that $\mathcal C_{\rm Tr}( \theta_1(P_{\mathbb C \hat 1}) ) = \mathcal C_{\rm Tr}( \theta_2(P_{\mathbb C \hat 1}) )$. 

Thus, $\theta_1(P_{\mathbb C \hat 1})$ and $\theta_2(P_{\mathbb C \hat 1})$ are Murray-von Neumann equivalent projections in $\mathcal M$. It follows that $\theta_1(\mathcal B(L^2(M, \tau)))$ and $\theta_2(\mathcal B(L^2(M, \tau)))$ are conjugate by a unitary in $\mathcal M$, and as $M$ is standardly represented on $L^2(M, \tau)$, we may then find such a unitary $u \in \mathcal M$ so that 
\[
u x u^* = u \theta_1(x) u^* = \theta_2(x) = x
\] 
for all $x \in M$, and hence $u \in M' \cap \mathcal M$.
\end{proof}

\begin{defn}
A von Neumann coupling between two finite, $\sigma$-finite von Neumann algebras $M$ and $N$ consists of a semi-finite von Neumann algebra $\mathcal M$, together with embeddings of $M$ and $N^{\rm op}$ into $\mathcal M$ such that $N^{\rm op} \subset M' \cap \mathcal M$ and such that each inclusion $M \subset \mathcal M$ and $N^{\rm op} \subset \mathcal M$ has a finite fundamental domain. 
\end{defn}

We use the notation $\mathcal M = {\prescript{}{M}{\mathcal M}}_N$ to indicate that $\mathcal M$ is a von Neumann coupling between $M$ and $N$. Two von Neumann couplings ${\prescript{}{M}{\mathcal M}}_N$ and ${ \prescript{}{M}{\mathcal N} }_N$ are isomorphic if there exists an isomorphism between $\mathcal M$ and $\mathcal N$ that restricts to the identity on $M$ and $N$, respectively.

\begin{prop}
Let $\mathcal M = {\prescript{}{M}{\mathcal M}}_N$ be a von Neumann coupling between $M$ and $N$. Fix a normal faithful semi-finite center-valued trace $\mathcal C_{\rm Tr}$ on $\mathcal M$, and consider the quantity
\[
[M : N]_{\mathcal M} := \mathcal C_{\rm Tr}(p)/\mathcal C_{\rm Tr}(q)
\] 
where $p$ and $q$ are rank-one projections in $\mathcal B(L^2(M))$ and $\mathcal B(L^2(N))$ respectively. 
Then this gives an invertible operator affiliated to $\mathcal Z(\mathcal M)$, which is independent of the choice of $\mathcal C_{\rm Tr}$ as well as the fundamental domains for $M \subset \mathcal B(L^2(M)) \subset \mathcal M$ and $N \subset \mathcal B(L^2(N)) \subset \mathcal M$. 
\end{prop}
\begin{proof}
Considering the decomposition $\mathcal M \cong \mathcal B(L^2(M)) \ovt P$ we see that non-zero projections in $\mathcal B(L^2(M))$ have central support equal to $1$. Therefore $\mathcal C_{\rm Tr}(p)/\mathcal C_{\rm Tr}(q)$ gives an invertible operator affiliated to $\mathcal Z(\mathcal M)$. 

As two faithful semi-finite center-valued traces are related by a positive injective operator affiliated to $\mathcal Z(\mathcal M)$, we see that the quantity $\mathcal C_{\rm Tr}(p)/\mathcal C_{\rm Tr}(q)$ is independent of $\mathcal C_{\rm Tr}$. Also, by Lemma~\ref{lem:conjfd} fundamental domains must be conjugate in $\mathcal M$ and hence the quantities ${\rm Tr}_{\mathcal M}(p)$ and ${\rm Tr}_{\mathcal M}(q)$ are each independent of the choice of fundamental domain.
\end{proof}

\begin{defn}
The quantity $[M : N]_{\mathcal M} \in \mathcal Z(\mathcal M)$ is the index of the coupling $\mathcal M$. The index group of $M$ is the subset of $\mathbb R_+^*$ consisting of all indices for factorial self-couplings of $M$ and is denoted by $\mathcal I_{vNE}(M)$. 
\end{defn}

Note that in Theorem~\ref{thm:vneindexgp} below, we justify the terminology by showing that the index group is indeed a subgroup of $\mathbb R_+^*$. 

Suppose $\mathcal M = {\prescript{}{M}{\mathcal M}}_N$ and $\mathcal N = { \prescript{}{N}{\mathcal M} }_Q$ are $M$-$N$ and $N$-$Q$ von Neumann couplings, respectively. Choose fundamental domains $\theta_{\mathcal M}: \mathcal B(L^2(N)) \to \mathcal M$ for $N^{\rm op} \subset \mathcal M$ and $\theta_{\mathcal N}: \mathcal B(L^2(N)) \to \mathcal N$ for $N \subset \mathcal N$.  Set $P_1 = \theta_{\mathcal M}( \mathcal B(L^2(N)) )' \cap \mathcal M$ and $P_2 = \theta_{\mathcal N}( \mathcal B(L^2(N)) )' \cap \mathcal N$. Then we have isomorphisms
\[
{\tilde \theta}_\mathcal M:  P_1 \ovt \mathcal B(L^2(N)) \to \mathcal M, \ \ \ \ \ \ \tilde \theta_{\mathcal N}: \mathcal B(L^2(N)) \ovt P_2 \to \mathcal N
\] 
such that ${\tilde \theta}_{\mathcal M}(a \otimes x) = a \theta_{\mathcal M}(x)$ and $\tilde \theta_{\mathcal N}(x \otimes b) = \theta_{\mathcal N}(x) b$ for $a \in P_1$, $b \in P_2$ and $x \in \mathcal B(L^2(N))$.

We then define the fusion (or composition) of the couplings ${\prescript{}{M}{\mathcal M}}_N$ and ${ \prescript{}{N}{\mathcal M} }_Q$ to consist of the von Neumann algebra 
\[
\mathcal M  \oovt{N}  \mathcal N := P_1 \ovt \mathcal B(L^2(N) ) \ovt P_2,
\] 
endowed with the  
embeddings of $M$ and $Q$ via the inclusions $\tilde \theta_{\mathcal M}^{-1} \times 1$ of $\mathcal M$ and $1 \times \tilde \theta_{\mathcal N}^{-1}$ of $\mathcal N$ given respectively by
\[
\mathcal M \ni x \mapsto {\tilde \theta}_{\mathcal M}^{-1}(x) \otimes 1 \in \mathcal M \oovt{N}  \mathcal N,
\ \ \ \ \ 
\mathcal N \ni x \mapsto 1 \otimes {\tilde \theta_{\mathcal N}}^{-1}(x) \in \mathcal M \oovt{N}  \mathcal N.
\]
Note that $\mathcal Z(\mathcal M) \subset P_1$ and $\mathcal Z(\mathcal N) \subset P_2$, so we have an inclusion 
\[
\mathcal Z(\mathcal M) \ovt \mathcal Z(\mathcal N) \subset \mathcal M \oovt{N}  \mathcal N.
\] 

\begin{prop}\label{prop:transitivevne}
Using the notation above, the von Neumann algebra $\mathcal M \oovt{N}  \mathcal N$ gives a von Neumann coupling between $M$ and $Q$ with index
\[
[M : Q]_{\mathcal M \oovt{N}  \mathcal N} = [M : N]_{\mathcal M} \otimes [ N : Q ]_{\mathcal N}.
\]
Moreover, up to isomorphism, this coupling is independent of the choice of fundamental domains for the inclusions $N \subset \mathcal M$ and $N \subset \mathcal N$. 
\end{prop}
\begin{proof}
We have 
\[
(\tilde \theta_{\mathcal M}^{-1} \times 1) (M) \subset (\tilde \theta_{\mathcal M}^{-1} \times 1) ({N^{\rm op}}' \cap \mathcal M) = P_1 \ovt N \ovt \mathbb C,
\] 
while 
\[
( 1 \times \tilde \theta_{\mathcal N}^{-1}) (Q) \subset ( 1 \times \tilde \theta_{\mathcal N}^{-1} ) (N' \cap \mathcal N) = \mathbb C \ovt  N^{\rm op} \ovt P_2,
\] 
so that the copies of $M$ and $Q$ in $\mathcal M \oovt{N}  \mathcal N$ commute. Since we have isomorphisms 
\[
\mathcal M \ovt P_2 \cong \mathcal M \oovt{N}  \mathcal N \cong P_1 \ovt \mathcal N,
\]
and since $P_1$ and $P_2$ are finite, we then have finite fundamental domains for $M$ and $Q$. We let $p \in \mathcal M$ and $q \in \mathcal N$ be minimal projections in fundamental domains for $M$ and $Q$ respectively. We also let $\mathcal C_i$ denote the center-valued trace on $P_i$ for $i = 1, 2$, and we define 
\[
\mathcal C_{\mathcal M} = \mathcal C_1 \otimes {\rm Tr}_{\mathcal B(L^2(N))} {\rm \ and \ } \mathcal C_{\mathcal N} = {\rm Tr}_{\mathcal B(L^2(N))} \otimes \mathcal C_2.
\] 
Then we have 
\begin{align}
[M : Q]_{\mathcal M \oovt{N}  \mathcal N}
&= (\mathcal C_1 \otimes {\rm Tr}_{\mathcal B(L^2(N))} \otimes \mathcal C_2 ) (\tilde \theta_{\mathcal M}^{-1}(p) \otimes 1)/  (\mathcal C_1 \otimes {\rm Tr}_{\mathcal B(L^2(N))} \otimes \mathcal C_2 ) (1 \otimes {\tilde \theta_{\mathcal N}}^{-1}(q)) \nonumber \\
&= {\mathcal C}_{\mathcal M}(p) \otimes \mathcal C_2(1)/ \mathcal C_1(1) \otimes {\mathcal C}_{\mathcal N}(q) \nonumber \\
&= [M : N]_{\mathcal M} \otimes [N : Q]_{\mathcal N}. \nonumber
\end{align}

Suppose now that we have fundamental domains for $N^{\rm op} \subset \mathcal M$ and $N \subset \mathcal N$ given respectively by $\phi_{\mathcal M}: \mathcal B(L^2(N)) \to \mathcal M$ and $\phi_{\mathcal N}: \mathcal B(L^2(N)) \to \mathcal N$.  We set 
\[
R_1 = \phi_{\mathcal M}(\mathcal B(L^2(N)))' \cap \mathcal M {\rm \ and \ } R_2 = \phi_{\mathcal N}(\mathcal B(L^2(N)))' \cap \mathcal N,
\] 
and we define the isomorphisms $\tilde \phi_{\mathcal M}$ and $\tilde \phi_{\mathcal N}$ as above. 

By Lemma~\ref{lem:conjfd} there exist unitaries $u \in {N^{\rm op}}' \cap \mathcal M$ and $v \in N' \cap \mathcal N$ so that $\phi_{\mathcal M} = {\rm Ad}(u) \circ \theta_{\mathcal M}$ and $\phi_{\mathcal N} = {\rm Ad}(v) \circ \theta_{\mathcal N}$. We then have $R_1 = u P_1 u^*$ and $R_2 = v P_2 v^*$. 

We consider the isomorphism $\alpha: R_1 \ovt \mathcal B(L^2(N)) \ovt R_2 \to P_1 \ovt \mathcal B(L^2(N)) \ovt P_2$ given by $\alpha = {\rm Ad}(u^*) \otimes {\rm id} \otimes {\rm Ad}(v^*)$. Under this isomorphism, the inclusion of $M$ coming from the fundamental domains $\phi_{\mathcal M}$ and $\phi_{\mathcal N}$ is given by $\alpha \circ ( \tilde \phi_{\mathcal M}^{-1} \times 1 )$ and still maps $M$ into $P_1 \ovt N \otimes \mathbb C$. Similarly, the new inclusion of $Q$ again maps into $\mathbb C \otimes N^{\rm op} \ovt P_2$.

If we restrict $\alpha$ to $P_1 \ovt \mathcal B(L^2(N))$ and consider the automorphism 
\[
\beta = \alpha \circ \tilde \phi_{\mathcal M}^{-1} \circ \tilde \theta_{\mathcal M} \in {\rm Aut}(P_1 \ovt \mathcal B(L^2(N)),
\] 
then for $a \in P_1$ and $x \in \mathcal B(L^2(N))$ we have
\begin{align}
\beta ( a \otimes x ) 
&= \alpha \circ \tilde \phi_ {\mathcal M}^{-1} ( a \theta_{\mathcal M}(x) ) \nonumber \\
&= \alpha \circ \tilde \phi_{\mathcal M}^{-1} ( u^* (u a u^*) \phi_{\mathcal M}(x) u ) \nonumber \\
&= \alpha \circ {\rm Ad}( \tilde \phi_{\mathcal M}^{-1} (u^*) ) ( u a u^* \otimes x ) \nonumber \\
&= {\rm Ad}(\alpha( \tilde \phi_{\mathcal M}^{-1} (u^*) ) ) (a \otimes x). \nonumber
\end{align}

Hence, $\beta = {\rm Ad}(\alpha( \tilde \phi_{\mathcal M}^{-1} (u^*) ) )$, and if we set $U =  \tilde \phi_{\mathcal M}^{-1}(u)$, then we see that $U \in R_1 \ovt N \subset R_1 \ovt \mathcal B(L^2(N))$ and  the map ${\rm Ad} ( U ) \circ  \alpha^{-1}$ intertwines the two inclusions of $M$ coming from the choice of fundamental domains. Similarly, if we set $V = \tilde \phi_{\mathcal N}^{-1}(v)$, then $V \in N^{\rm op} \ovt R_2 \subset \mathcal B(L^2(N)) \ovt R_2$ and the map ${\rm Ad} ( U ) \circ  \alpha^{-1}$ intertwines the two inclusions of $Q$ coming from the fundamental domains. We then have that $U$ and $V$ commute and the isomorphism ${\rm Ad}(UV) \circ \alpha^{-1}$ intertwines both the inclusions of $M$ and $Q$. 
\end{proof}

\begin{defn}
Two finite, $\sigma$-finite von Neumann algebras $M$ and $N$ are von Neumann equivalent, denoted $M \sim_{vNE} N$, if there exists a von Neumann coupling between them. 
\end{defn}

Von Neumann equivalence is indeed an equivalence relation. Reflexivity follows by considering the trivial von Neumann coupling $\mathcal B(L^2(M, \tau))$ with the standard embeddings of $M$ and $M^{\rm op}$. If ${\prescript{}{M}{\mathcal M}}_N$ is a von Neumann coupling between $M$ and $N$, then ${\mathcal M}^{\rm op}$ gives a von Neumann coupling between $N$ and $M$ (with index $[N: M]_{{\mathcal M}^{\rm op}} = [M:N]_{\mathcal M}^{-1}$), showing that this relation is symmetric, while transitivity of this relation follows from Proposition~\ref{prop:transitivevne}. We also see that by considering the index it follows that $\mathcal I_{vNE}(M)$ is a subgroup of $\mathbb R_+^*$. We record all these facts in the following theorem.

\begin{thm}\label{thm:vneindexgp}
Von Neumann equivalence gives an equivalence relation on the collection of finite, $\sigma$-finite von Neumann algebras. Given a finite, $\sigma$-finite von Neumann algebra $M$, we have that $\mathcal I_{vNE}(M)$ is a subgroup of $\mathbb R_+^*$, which only depends on the von Neumann equivalence class of $M$. 
\end{thm} 

Two finite factors $M$ and $N$ are virtually isomorphic if there exists a normal Hilbert $M$-$N$-bimodule $\mathcal H$ that has finite von Neumann dimension as both an $M$ module and an $N$ module. Two ICC groups $\Gamma$ and $\Lambda$ are virtually $W^*$-equivalent if $L\Gamma$ and $L\Lambda$ are virtually isomorphic. The notion of virtual isomorphism of factors was first studied by Popa in \cite[Section 1.4]{Po86}, while the terminology was coined more recently in \cite[Section 4.1]{PoSh18}. 

\begin{thm}
Let $M$ and $N$ be finite factors, and suppose that $\mathcal H$ is a Hilbert $M$-$N$-bimodule that is finite as both an $M$-module and an $N$-module. Then $M$ and $N$ are von Neumann equivalent, and a von Neumann coupling $\mathcal M$ may be chosen so that 
\[
[M : N]_{\mathcal M} = \dim_M(\mathcal H) \cdot \dim_{N}(\mathcal H)^{-1}.
\]
\end{thm}
\begin{proof}
Let $\mathcal H$ be a Hilbert $M$-$N$-bimodule that is finite as both an $M$-module and an $N$-module. Let $R$ denote the hyperfinite II$_1$ factor and set $\mathcal M = R \ovt \mathcal B(\mathcal H)$. 

Suppose $t = \dim_M(\mathcal H) < \infty$, and take $k \in \mathbb N$ so that $k > t$. If we take a projection $p \in \mathbb M_n(\mathbb C) \ovt M^{\rm op}$ such that $( {\rm Tr} \otimes \tau ) (p ) = t$, then we have an isomorphism of inclusions between $M \subset \mathcal B(\mathcal H)$ and $p M \subset p ( \mathbb M_n(\mathbb C) \ovt \mathcal B(L^2(M, \tau)) ) p$.

If we now take a projection $q \in R$ so that $\tau(q) = t/n$, then we have that $q$ and $p$ are equivalent projections in $R \ovt \mathbb M_n(\mathbb C) \ovt M^{\rm op}$ and hence we see that we have an isomorphism of inclusions between $M \subset R \ovt \mathcal B(\mathcal H)$ and $q M \subset q R q \ovt \mathbb M_n(\mathbb C) \ovt \mathcal B(L^2(M, \tau))$. In particular, we then see that we have a fundamental domain for the inclusion $M \subset R \ovt \mathcal B(\mathcal H)$. Moreover, the trace of a rank-one projection in this fundamental domain will be $n \tau(q) = t$. 

We similarly see that the inclusion $N \subset R \ovt \mathcal B(\mathcal H)$ has a fundamental domain, and the trace of a rank-one projection in its fundamental domain will be $\dim_{N}(\mathcal H)$. Thus, $\mathcal M$ is a von Neumann coupling between $M$ and $N$ with index $\dim_M(\mathcal H)  \dim_{N}(\mathcal H)^{-1}.$
\end{proof}

We note that Pimsner and Popa showed in \cite[Proposition 1.11]{PP86} that property Gamma of Murray and von Neumann \cite{MuvN43} is a virtual isomorphism invariant, while Theorem~\ref{thm:vnequivalencegp} together with Effros's Theorem \cite{Ef75} and Caprace's example \cite[Section 5.C]{DuTDWe18} show that property Gamma is not an invariant of von Neumann equivalence. 

\begin{cor}
If $M$ is a II$_1$ factor and $s, t > 0$, then $M^t$ and $M^s$ have a von Neumann coupling $\mathcal M$ that satisfies 
\[
[M^t: M^s]_{\mathcal M} = t^2/s^2.
\] 
Consequently, we have an inclusion $\mathcal F(M)^2 \subset \mathcal I_{vNE}(M)$. 
\end{cor}

We may now show the relationship between von Neumann equivalence for groups and for finite von Neumann algebras as stated in Theorem~\ref{thm:vnequivalencegp}. 

\begin{proof}[Proof of Theorem~\ref{thm:vnequivalencegp}]
We first suppose that $\mathcal M$ is an $L\Gamma$-$L\Lambda$ von Neumann coupling. If $p \in \mathcal Z(\mathcal M)$ is a non-trivial central projection, then $p \mathcal M$ is also an $L\Gamma$-$L\Lambda$ von Neumann coupling, hence we may assume that $\mathcal M$ is $\sigma$-finite and fix a semi-finite normal faithful trace ${\rm Tr}$ on $\mathcal M$. 

We identify $\Gamma$ (resp.\ $\Lambda$) as a subgroup of $\mathcal U(L\Gamma)$ (resp.\ $\mathcal U(L\Lambda)$) and then consider the commuting trace-preserving actions of $\Gamma$ and $\Lambda$ on $\mathcal M$ given by conjugation. If we have a fundamental domain $L\Gamma \subset \mathcal B(\ell^2 \Gamma) \subset \mathcal M$, then the rank-one projection onto the subspace spanned by $\delta_e \in \ell^2 \Gamma$ gives a finite-trace fundamental domain for the conjugation action of $\Gamma$ on $\mathcal M$. We similarly have a finite-trace fundamental domain for the action of $\Lambda$ on $\mathcal M$, and hence we see that $\mathcal M$ is then a $\Gamma$-$\Lambda$ von Neumann coupling.

Now suppose that $(\mathcal M, {\rm Tr})$ is a $\Gamma$-$\Lambda$ von Neumann coupling. We set $\mathcal N = \mathcal M \rtimes (\Gamma \times \Lambda)$. We then have embeddings $L\Gamma, L\Lambda \subset \mathcal N$. A $\Gamma$-fundamental domain in $\mathcal M$ gives a $\Gamma$-equivariant embedding $\ell^\infty \Gamma \subset \mathcal M$ and hence we get an embedding of von Neumann algebras 
\[
\mathcal B(\ell^2 \Gamma) \cong \ell^\infty \Gamma \rtimes \Gamma \subset \mathcal M \rtimes \Gamma \subset \mathcal N.
\] 
Thus $\mathcal N$ has an $L \Gamma$ fundamental domain. Moreover, if $P_e$ is the rank-one projection onto the span of $\delta_e \in \ell^2 \Gamma$, then we have $P_e \in \mathcal M \subset \mathcal N$, and therefore the fundamental domain for $L\Gamma$ has finite trace and so must be finite. We similarly have a finite-trace fundamental domain for $L\Lambda$ in $\mathcal N$, and hence $\mathcal N$ is an $L\Gamma$-$L\Lambda$ von Neumann coupling.
\end{proof}

The analogue of the index group has also been considered in the setting of measure equivalence. For instance, in \cite[Section 2.2]{Ga02} or \cite[Question 2.8]{Ga05} Gaboriau considered the set of indices of all ergodic self measure equivalence couplings of a group $\Gamma$. For minimally almost periodic groups \cite{vNWi40} any non-trivial ergodic probability measure-preserving action is weak mixing, and a simple argument then shows that the composition of two ergodic measure equivalence self-couplings is again ergodic. This then shows that for minimally almost periodic groups, the set of indices of all ergodic self measure equivalence couplings is a subgroup of $\mathbb R_+^*$. It is not clear, however, if this set is a group in general, or that it is a measure equivalence invariant, as the composition of ergodic measure equivalence couplings need not be ergodic in general. For ICC groups, at least, we have the following relationship between indices for ergodic measure equivalence couplings and the index group for the group von Neumann algebra:

\begin{prop}\label{prop:meselfcoupling}
Suppose $(\Omega, m)$ is an ergodic ME-self-coupling of an ICC group $\Gamma$, then $[\Gamma: \Gamma]_\Omega \in \mathcal I_{vNE}(L\Gamma)$. 
\end{prop} 
\begin{proof}
We see from the proof of Theorem~\ref{thm:vnequivalencegp} that if $(\Omega, m)$ is an ergodic measure equivalence self-coupling of $\Gamma$, then $L^\infty(\Omega, m) \rtimes (\Gamma \times \Gamma)$ is a von Neumann self-coupling for $L\Gamma$, and if $L^\infty(\Omega, m) \rtimes (\Gamma \times \Gamma)$ is a factor, then the indices for these couplings agree. Thus it suffices to show that under these hypotheses, we have that $L^\infty(\Omega, m) \rtimes (\Gamma \times \Gamma)$ is a factor. 

If we let $\Gamma_i$ denote the $i$th copy of $\Gamma$ in $\Gamma \times \Gamma$, then Proposition~\ref{prop:diagonal} shows that the fundamental domain for $\Gamma_1$ leads to an isomorphism 
\[
L^\infty(\Omega, m) \rtimes (\Gamma \times \Gamma) \cong (L^\infty( \Omega/\Gamma_1 ) \rtimes \Gamma_2 ) \ovt \mathcal B(\ell^2 \Gamma_1).
\] 
Since $\Gamma_2$ is ICC, and since $\Gamma_2 \actson \Omega/ \Gamma_1$ is an ergodic and measure-preserving action on a finite measure space, Murray and von Neumann's proof of factoriality of $L\Gamma_2$ \cite{MuvN43} shows that we have 
\[
L\Gamma_2' \cap ( L^\infty( \Omega/\Gamma_1 ) \rtimes \Gamma_2 ) \subset L^\infty(\Omega/\Gamma_1)^{\Gamma_2} = \mathbb C.
\] 
Hence $L^\infty( \Omega/\Gamma_1 ) \rtimes \Gamma_2$ is a factor, and so is 
\[
(L^\infty( \Omega/\Gamma_1 ) \rtimes \Gamma_2 ) \ovt \mathcal B(\ell^2 \Gamma_1) \cong L^\infty(\Omega, m) \rtimes (\Gamma \times \Gamma). 
\] 
\end{proof}

\vskip -.1in

In \cite{PoVa10} Popa and Vaes study the collection $\mathcal S_{\rm eqrel}(\Gamma)$ of fundamental groups for equivalence relations associated to free, ergodic, probability measure-preserving actions of $\Gamma$. Each element in such a fundamental group gives rise to an ergodic measure equivalence coupling with the same index \cite[Theorem 3.3]{Fu99A}, and hence we obtain the following corollary. 

\begin{cor}\label{cor:subgpindex}
For a countable ICC group $\Gamma$, we have $\mathcal F < \mathcal I_{vNE}(L\Gamma)$ for all $\mathcal F \in \mathcal S_{\rm eqrel}(\Gamma)$.
\end{cor}

As an example, the previous corollary applied to Theorem 1.3 in \cite{PoVa10} shows that for $n \geq 3$ we have $\mathbb Q_+^* < \mathcal I_{vNE}(L (\mathbb Z^n \rtimes SL(n, \mathbb Z) ) )$.

\section{Open problems}

While we have introduced a basic framework for studying von Neumann equivalence, there are many problems left open and further directions to explore. Since our preprint appeared, results were obtained in \cite{Is21, Ba23, Ba22} showing that von Neumann equivalence for groups also preserves weak amenability, the Cowling-Haagerup constant, the weak Haagerup property, the approximation property, exactness, and for $d \geq 2$, certain $M_d$-versions of the approximation property, weak amenability, and the weak Haagerup property. In this section we list some additional problems regarding von Neumann equivalence whose solutions would further illuminate this concept.

\begin{problem}
If $M$ is von Neumann equivalent to a factor with fundamental group $\mathbb R_+^*$, then we have $\mathcal I_{vNE}(M) = \mathbb R_+^*$. It would be interesting to have examples of von Neumann algebras, or even groups $\Gamma$, such that $\mathcal I_{vNE}(L\Gamma)$ is not $\mathbb R_+^*$, perhaps trivial. Or, in view of Proposition~\ref{prop:meselfcoupling}, examples when $\mathcal I_{vNE}(L\Gamma)$ is non-trivial and discrete.
\end{problem}

\begin{problem}
Find an example of a group $\Gamma$ such that a group $\Lambda$ is von Neumann equivalent to $\Gamma$ if and only if $\Lambda$ is commensurable up to finite kernel with $\Gamma$. 
\end{problem}

\begin{problem}
Chifan and Ioana in \cite{ChIo11} gave examples of groups that are orbit equivalent (and hence also von Neumann equivalent) but that are not $W^*$-equivalent. Popa and Shlyakhtenko showed in \cite[Propostion 4.3]{PoSh18} that these are not even virtually $W^*$-equivalent (and additional examples with this property are also given). Combining this result with Theorem~\ref{thm:vnequivalencegp} shows that von Neumann equivalence for groups (resp.\ for von Neumann algebras) is strictly coarser than virtual $W^*$-equivalence (resp.\ virtual isomorphism). The related problem of finding groups that are $W^*$-equivalent but not measure equivalent remains open (see \cite{ChIo11}). A natural related problem is then to find examples of groups that are von Neumann equivalent but not measure equivalent.
\end{problem}

\begin{problem}
A natural candidate for an example in the previous problem would be a group $\Gamma$ such that $\Gamma$ is von Neumann equivalent to $\Gamma \times \mathbb Z$, but such that $\Gamma$ is not measure equivalent to $\Gamma \times \mathbb Z$. The class of groups $\Gamma$ that are ME-stable, i.e. that are measure equivalent to $\Gamma \times \mathbb Z,$ is studied some in \cite{Ki15} or \cite{KiTD20}. It would be interesting to study the class of groups or von Neumann algebras that are von Neumann equivalence stable. 
\end{problem}

\begin{problem}
Let $\mathcal C$ be a class of von Neumann algebras with semi-finite traces that contains $\ell^\infty I$ (with the counting measure) for any set $I$, and that is stable under taking opposites, tensor products, and fixed point subalgebras for actions that have fundamental domains. We may then define two groups $\Gamma$ and $\Lambda$ to be $\mathcal C$-equivalent if they have commuting trace-preserving actions on a von Neumann algebra $\mathcal M \in \mathcal C$ such that each group has a finite-trace fundamental domain. The same argument in Section~\ref{sec:VNE} shows that being $\mathcal C$-equivalent is an equivalence relation. When $\mathcal C$ contains all abelian von Neumann algebras we recover measure equivalence, while when $\mathcal C$ is the class of all semi-finite von Neumann algebras we obtain von Neumann equivalence. 

One may also consider other classes, for example, the class of all finite von Neumann algebras, or the class of all semi-finite injective von Neumann algebras. It would be interesting to know when two such classes give the same, or different, equivalence relations.  
\end{problem}

\begin{problem}
There are additional powerful ME-invariants that seem extremely difficult to adapt to the noncommutative setting. In particular, it would be interesting to know if the ratio of $\ell^2$-Betti numbers is a von Neumann equivalence invariant \cite{Ga00}. In the setting of $W^*$-equivalence this is closely connected to the free group factor problem. 
\end{problem}

\begin{problem}
Can von Neumann equivalence be used to induce bounded cohomology, analogous to the situation for measure equivalence in \cite{MoSh06}?
\end{problem}

\begin{problem}
It is shown in \cite{Sa09} that biexactness is a measure equivalence invariant. More recently, it is shown in \cite{DiP23} that biexactness is also a $W^*$-equivalence invariant. This suggests that it is very likely that biexactness is also an invariant of von Neumann equivalence.
\end{problem}

\begin{problem}
Is the property of being coarsely embeddable into Hilbert space a von Neumann equivalence invariant?
\end{problem}

\begin{problem}
Measure equivalence for locally compact groups has been introduced in \cite{KoKyRa21, KoKyRa21b}. There should also be a notion of von Neumann equivalence for locally compact groups. Given the noncommutative setting, it may also be possible to introduce the notion of von Neumann equivalence for compact quantum groups. 
\end{problem}

\begin{problem}
Since there is such a simple classification of finite injective von Neumann algebras \cite{Co76}, it seems interesting to determine which of these are von Neumann equivalent.
\end{problem}

\appendix

\section{Measure equivalence and properly proximal groups}

For the benefit of the reader who may be less familiar with von Neumann algebras, we include here a separate proof that proper proximality is a measure equivalence invariant. We refer the reader to \cite{Zi84} or \cite{Fu11} for preliminary results on measure equivalence and cocycles.

If $E_*$ is a separable Banach space and $(X, \mu)$ is a standard Borel space, then we denote by $L^1(X; E_*)$ the set of norm-integrable Borel functions from $X$ to $E_*$, where we identify two functions if they agree almost everywhere. This is naturally a Banach space with norm $\| f \| = \int \| f(x) \| \, d\mu$. We set $E = (E_*)^*$ and let $L^\infty(X; E)$ denote the space of measurable, essentially bounded functions from $X$ to $E$, where $E$ is given the Borel structure coming from the weak$^*$-topology, and we identify functions that agree almost everywhere. We have a natural identification of $L^\infty(X; E)$ with $L^1(X; E_*)^*$ via the pairing $\langle \varphi, f \rangle = \int \varphi_x(f_x) d\mu(x)$. If $K \subset E$ is a weak$^*$-compact convex subset, then $L^\infty(X; K)$ gives a weak$^*$-compact convex subset of $L^\infty(X; E)$.

If $E$ is a dual Banach $\Lambda$-module and $K \subset E$ is a non-empty weak$^*$-compact convex $\Lambda$-invariant subset, $\Gamma \actson (X, \mu)$ is a  probability measure-preserving action, and  $\alpha: \Gamma \times X \to \Lambda$ is a cocycle, then we obtain an induced affine action of $\Gamma$ on $L^\infty(X; K)$ by 
\[
(\gamma \cdot f)(x) = \alpha(\gamma, \gamma^{-1}x) f(\gamma^{-1} x).
\]

Let $(\Omega, m)$ be an ME-coupling of two groups $\Gamma$ and $\Lambda$, and let $X \subset \Omega$ be a fundamental domain for the $\Lambda$-action. Under the identification $\Omega/ \Lambda \cong X$ given by $\Lambda \omega \mapsto \Lambda \omega \cap X$, the action $\Gamma \actson \Omega/ \Lambda$ translates to 
\[
\gamma  \cdot x =  \alpha(\gamma, x) \gamma x,
\]
where $\alpha$ is the Zimmer cocycle, which is defined by the property that $\alpha(\gamma, x)$ is the unique element in $\Lambda$ such that $\alpha(\gamma, x) \gamma x \in X$.

The following result is well known.

\begin{prop}
There exists a $\Gamma$-fixed point in $L^\infty(\Omega/\Lambda; K)$ if and only if there exists a $\Lambda$-fixed point in $K$.
\end{prop}
\begin{proof}
Suppose $\xi: \Omega/ \Lambda \to K$ is a Borel map that satisfies $\xi(\gamma x) = \alpha(\gamma, x) \xi(x)$. We then define the map $\tilde \xi: \Omega \to K$ by $\tilde \xi(\lambda, x) = \lambda^{-1} \xi(x)$, where we identify $\Omega$ with $\Lambda \times (\Omega/ \Lambda)$. Then as $\tilde \xi$ is invariant under the induced $\Gamma$-action and is equivariant with respect to the $\Lambda$-action, we therefore obtain a $\Lambda$-equivariant map from $\Omega/ \Gamma \to K$. Integrating this map with respect to the $\Lambda$-invariant measure on $\Omega/ \Gamma$ then gives a $\Lambda$-fixed point. 
\end{proof}

We recall from Proposition~\ref{prop:properlyproximalequivalent} that a group $\Lambda$ is properly proximal if there exists a dual Banach $\Lambda$-module $E$ and a non-empty weak$^*$-compact convex $\Lambda$-invariant subset $K \subset E$ such that $K$ has a properly proximal point, but has no fixed point.

A cocycle $\alpha: \Gamma \times X \to \Lambda$ is proper if for all $\varepsilon > 0$ and $F \subset \Lambda$ finite, there exists $F' \subset \Gamma$ finite such that $\mu( \{ x \mid \alpha(\gamma, \gamma^{-1} x) \in F \} ) < \varepsilon$ for all $\gamma \in \Gamma \setminus F'$. It's easy to see that a cocycle coming from an ME-coupling is proper.

\begin{prop}
If the action $\Lambda \actson K$ is properly proximal, and if the cocycle $\alpha$ is proper, then the induced action $\Gamma \actson L^\infty(X; K)$ is properly proximal. 
\end{prop}
\begin{proof}
We assume for simplicity that $K$ is contained in the unit ball of $E^*$. Fix $k \in K$ such that for all $h \in \Lambda$ we have $\lim_{\lambda \to \infty} \lambda h k - \lambda k = 0$. We view $k \in L^\infty(X; K)$ as a constant function. Fix $g \in \Gamma$, $\varepsilon > 0$, and $\mathcal F \subset L^1(X; E)$ a finite collection of step functions with finite range $F_0$ contained in the unit ball of $E$. Fix a set $X_0 \subset X$ such that $\mu(X_0) > 1 - \varepsilon$ and such that $x \mapsto \alpha(g, x)$ ranges in a finite set $F_{00} \subset \Lambda$. 

Since $k$ is a convergence point for $\Lambda$, there exists a finite set $F_{00}' \subset \Lambda$ such that for all $\lambda \in \Lambda \setminus F_{00}'$ we have $| \langle \lambda h k - \lambda k, a \rangle | < \varepsilon$ for all $h \in F_{00}, a \in F_0$. As the cocycle $\alpha$ is proper, there exists a finite set $G_0 \subset \Gamma$, so that if $E_\gamma = \{ x \in X \mid \alpha(\gamma, \gamma^{-1} x) \not\in F_{00}' \}$, then $\mu(E_\gamma) > 1- \varepsilon$ for all $\gamma \in \Gamma \setminus G_0$. For $\gamma \in \Gamma \setminus G_0$, and $f \in \mathcal F$ we then have 
\begin{align}
| \langle \gamma g k - \gamma k, f \rangle |
&= \left| \int  \langle \alpha(\gamma g, g^{-1} \gamma^{-1} x) k - \alpha(\gamma, \gamma^{-1} x) k, f(x) \rangle d\mu(x) \right| \nonumber \\
&\leq \int | \langle \alpha(\gamma, \gamma^{-1} x) \alpha(g, g^{-1} \gamma^{-1} x) k - \alpha(\gamma, \gamma^{-1}x) k, f(x) \rangle | d\mu(x) \nonumber \\
&\leq \varepsilon + \int_{\gamma g X_0} \sup_{h \in F_{00}} \sup_{a \in F_0} | \langle \alpha(\gamma, \gamma^{-1}x) h k - \alpha(\gamma, \gamma^{-1}x)k, a \rangle | d\mu(x) \nonumber \\
&\leq 2\varepsilon + \int_{\gamma g X_0 \cap E_\gamma} \sup_{h \in F_{00}} \sup_{a \in F_0} | \langle \alpha(\gamma, \gamma^{-1}x) h k - \alpha(\gamma, \gamma^{-1}x)k, a \rangle | d\mu(x) 
< 3 \varepsilon. \nonumber
\end{align}
Since simple functions are dense in $L^1(X; E)$, it follows that $k$ is a convergence point for the action $\Gamma \actson L^\infty(X; K)$. 
\end{proof}

\begin{cor}
If two groups $\Gamma$ and $\Lambda$ are measure equivalent, then $\Gamma$ is properly proximal if and only if $\Lambda$ is properly proximal. 
\end{cor}

\providecommand{\bysame}{\leavevmode\hbox to3em{\hrulefill}\thinspace}
\providecommand{\MR}{\relax\ifhmode\unskip\space\fi MR }

\providecommand{\MRhref}[2]{%
  \href{http://www.ams.org/mathscinet-getitem?mr=#1}{#2}
}
\providecommand{\href}[2]{#2}

\end{document}